\documentclass[11pt,epsfig,amsfonts]{amsart}
\usepackage{stmaryrd}
\usepackage{xcolor}
\usepackage{todonotes}
\usepackage{epsfig}
\usepackage{enumitem}
\usepackage{comment}
\usepackage{amsmath}
\usepackage{amssymb}
\usepackage{amscd}
\usepackage{graphicx}
\usepackage{lineno}
\usepackage{pstricks}
\usepackage{tocvsec2}
\usepackage{mathtools}
\usepackage[basic]{mnotes}
\allowdisplaybreaks[4]

\usepackage [backref, colorlinks, citecolor=red, anchorcolor=black, linkcolor=blue]{hyperref}

\usepackage{amsfonts}
\usepackage[top=35mm, bottom=35mm, left=30mm, right=30mm]{geometry}
\usepackage{verbatim}
\usepackage{graphicx, amssymb, amsmath,amsthm}
\allowdisplaybreaks[4]
\usepackage{appendix}
\usepackage{mathrsfs}
\usepackage{cite}
\usepackage{titlesec}
\usepackage{titletoc}
\titleformat{\part}[block]{\huge\bfseries\centering}{Part \arabic{part}.}{0.5em}{}[]
\titleformat{\section}[block]{\Large\bfseries\centering}{ \arabic{section}.}{1em}{}[]
\titleformat{\subsection}[block]{\mdseries\bfseries}{\arabic{section}.\arabic{subsection}.}{1em}{}[]
\titleformat{\subsubsection}[block]{\normalsize\bfseries}{\arabic{section}.\arabic{subsection}.\arabic{subsubsection}.}{1em}{}[]

\makeatletter
\@namedef{subjclassname@2020}{%
	\textup{2020} Mathematics Subject Classification}
\makeatother

\def \a{\alpha} \def \b{\beta} \def \g{\gamma} \def \d{\delta}
  \def \p{\Theta} 
  \def \z{\zeta} 
\def \O{\Omega}   \def \r{\rho}
  \def \G{\Gamma}\def\rank{\mr{rank}}
 \def \spa{\mathrm{span}}
 
\def\Fix{\mr{Fix}}
\def\mr{\mathrm}
\def \Z{\mathbb Z}
\def\mc{\mathcal}\def\id{\mr{id}}\def\mb{\mathbb}\def\id{\mathrm{id}}\def\mr{\mathrm }\def\CP{\mb{CP}}\def\C{\mb C}\def\R{\mb R}\def\Pl{\mr{Pl}}\newcommand{\HP}{\mathbb{HP}}\def\PGL{\mr{PGL}}\def\Hom{\mr{Hom}}
\newcommand{\Gr}{\mathrm{Gr}}\def\RS{\mr{SR}(\O,\mb H)}

\newcommand{\GL}{\mathrm{GL}}\newcommand{\Mat}{\mathrm{Mat}}
\usepackage{xcolor}

\usepackage{xcolor}
\newtheorem{thm}{Theorem}[section]
\newtheorem{lem}[thm]{Lemma}
\newtheorem{cor}[thm]{Corollary}
\newtheorem{prop}[thm]{Proposition}

\numberwithin{equation}{section}
\usepackage{epsfig}

\newtheoremstyle{remarkbold}
  {6pt plus 2pt minus 2pt} 
  {6pt plus 2pt minus 2pt} 
  {\normalfont}
  {}
  {\bfseries}
  {.}
  {.5em} 
  {}

\theoremstyle{remarkbold}
\newtheorem{defn}[thm]{Definition}
\newtheorem{rem}[thm]{Remark}

\newtheorem{exam}[thm]{Example}

\dottedcontents{part}[0.5em]{\normalsize}{1em}{4pt}
\dottedcontents{section}[1.5em]{\normalsize}{1.5em}{4pt}
\dottedcontents{subsection}[3.5em]{\small}{2.0em}{4pt}
\dottedcontents{subsubsection}[4em]{\normalsize}{3.0em}{4pt}
\numberwithin{equation}{section}

\renewcommand*{\backref}[1]{}
\renewcommand*{\backrefalt}[4]{\quad \tiny
	\ifcase #1 (\textbf{NOT CITED.})%
	\or    (Cited on Section~#2.)%
	\else   (Cited on Section~#2.)%
	\fi}

\makeatletter

\def\MRbibitem{\@ifnextchar[\my@lbibitem\my@bibitem}

\def\mybiblabel#1#2{\@biblabel{{\hyperref{http://www.ams.org/mathscinet-getitem?mr=#1}{}{}{#2}}}}

\def\myhyperanchor#1{\Hy@raisedlink{\hyper@anchorstart{cite.#1}\hyper@anchorend}}

\def\my@lbibitem[#1]#2#3#4\par{%
	\item[\mybiblabel{#2}{#1}\myhyperanchor{#3}\hfill]#4%
	\@ifundefined{ifbackrefparscan}{}{\BR@backref{#3}}%
	\if@filesw{\let\protect\noexpand\immediate
		\write\@auxout{\string\bibcite{#3}{#1}}}\fi\ignorespaces%
}

\def\my@bibitem#1#2#3\par{%
	\refstepcounter\@listctr
	\item[\mybiblabel{#1}{\the\value\@listctr}\myhyperanchor{#2}\hfill]#3%
	\@ifundefined{ifbackrefparscan}{}{\BR@backref{#2}}%
	\if@filesw\immediate\write\@auxout
	{\string\bibcite{#2}{\the\value\@listctr}}\fi\ignorespaces%
}

\makeatother


	\author{Chunlin Liu}
	\address[C. Liu]
	{School of Mathematical Sciences, Dalian University of Technology, Dalian, 116024, P.R. China, and Institute of Mathematics, Polish Academy of Sciences, ul. Śniadeckich 8, 00-656 Warszawa, Poland}
	\email{chunlinliu@mail.ustc.edu.cn}
		\author{Giovanni Moreno}
	\address[G. Moreno]
	{Department of Mathematical Methods in Physics,
    Department of Physics,
    Warsaw University,
    ul. Pasteura 5, 02-093 Warsaw, Poland.}
	\email{Giovanni.Moreno@fuw.edu.pl}
    	\author[Haipan Shi]{Haipan Shi$^{*}$}
	\address[H. Shi]
	{School of Science,
	Huzhou Normal University, Wuxing District, Huzhou City, Zhejiang Province, 313000, P. R. China }
	\email{03298@huznu.edu.cn}
\thanks{$^{*}$Corresponding author.Email:03298@huznu.edu.cn}

\newcommand{\1}{\underline 1}
\renewcommand{\epsilon}{\varepsilon}
\renewcommand{\epsilon}{\varepsilon}

\begin{document}
\title[Classifying Slice-Regular Polynomials]{Classifying Slice-Regular Polynomials via Group Actions on the Twistor Space}
\maketitle
\begin{abstract}
    We study the equivalence classes of slice--regular functions $f:\Omega\to\mathbb{H}$ on a symmetric slice domain $\Omega$, and of their subclass made of polynomial slice--regular functions, with respect to the natural action of $\PGL(2,\mathbb{H})$ and its subgroups, by employing the twistor construction. In particular, we characterize slice--regular functions whose twistor lift is planar and belongs to a given orbit, and we find normal classes of slice--regular polynomials with respect to the action of a parabolic subgroup of $\GL(2,\mathbb{H})$.
\end{abstract}

\tableofcontents
\section{Introduction}

The theory of slice regular functions, introduced in 2006--2007 by G.~Gentili and D.~C.~Struppa \cite{GGSD06,GGS07}, has since its inception experienced several major development: see, for instance,   \cite{ColomboGentiliSabadiniStruppa2009,GhiloniPerotti2011, AltavillaSarfatti2019} and references therein. One of the main selling points of the definition of a slice--regular function is that it encompasses all   convergent power series of
the form
$$
\sum_{n=0}^{\infty}q^na_n
$$
of a quaternionic variable $q$ with quaternionic coefficients $a_n$.\par 
Let $\mathbb{S}$ be the unit sphere of imaginary units in $\mathbb{H}$. 
A fundamental milestone in the development of the theory came in 2014 with the introduction of the twistor transform of a slice--regular function by G.~Gentili, S.~Salamon and C.~Stoppato \cite{GSS14}: building on the natural link between twistor geometry and Lie sphere geometry \cite{HoffmannShapiro2013}, they observed that a slice--regular function $f$ maps spheres $x+y\cdot\mathbb{S}$ in $\mathbb{H}\smallsetminus\mathbb{R}$ to other spheres of $\mathbb{H}$, and such a correspondence is nothing but the twistor transform of $f$, see also~\cite{altavilla2018twistor}.\par 
In this paper we exploit the natural group actions on $\Gr_2(\C^4)$ to introduce equivalence relations on the set of slice--regular functions $\RS$ on a symmetric slice domain $\Omega\subset\mathbb{H}$.\par 
Our work relates to the general problem of finding invariants of slice--regular function, that has been tackled, with different techniques, by various authors. For instance, C.~Bisi and J.~Winkelmann have recently studied the fundamental invariants of slice--regular functions with respect to the autumorphism groups of (complexified) quaternions \cite{BisiWinkelmann2026}. A special class of slice--regular functions appears also in the paper  \cite{AltavillaDeFabritiis2018} by A.~Altavilla and C.~de Fabritiis, where the authors characterize the property of preserving one complex slice.\par
Our approach relies on the homogeneous structure of the Grassmann manifold $\Gr_2(\C^4)$, that is the target space of the twistor lifts of slice--regular functions. A noteworthy example that exploits such an approach  is the work \cite{jiao2003rigidity} by X.~Jiao and J.~Peng, prior to the introduction of the idea of a slice--regular function: the authors uses the first and second fundamental forms to classify  holomorphic curves in complex Grassmann manifolds, see also~\cite{jiao2020orthogonal,jiao2022minimal,zhang2021minimal}.\par
More accurately, we study the orbits in the Pl\"ucker embedding space $\mathbb{P}(\Lambda^2\C^4)$ of $\Gr_2(\mathbb{C}^4)$ with respect to the natural action of $\PGL(2,\mathbb{H})$ (Theorem~\ref{thm:pole-orbit-types}) and then we use this result to classify \emph{planar} slice--regular functions, i.e., elements of $\RS$, whose twistor lift lies inside a hyperplane section of $\mathbb{P}(\Lambda^2\C^4)$ of a given \emph{type}. 
\subsubsection*{Structure of the paper}
In Section~\ref{secPrel} we cover all necessary definitions an theorems. We fix the notations and the conventions for the complex and quaternionic variables and the related linear algebra, with a particular emphasis on $\GL(2,\mathbb{H})$ and $\GL(4,\mathbb{C})$. We recall the structure of  $\Gr_2(\mathbb{C}^4)$, its minimal projective embedding in $\mathbb{P}(\Lambda^2\C^4)$, its local affine charts, and the notion of a hyperplane section. We review the twistor construction allowing to lift a slice--regular function to a holomorphic curve with values in  $\Gr_2(\mathbb{C}^4)$: this leads naturally to the notion of \emph{planarity} and \emph{hyperplane type} of a slice--regular function, see Definition~\ref{def:hyperplane-sectional} and Definition~\ref{def:A-planar-unified}, respectively.\par
In Section~\ref{secHyperPlaneSections} we carry out a full classifications of $\PGL(2,\mathbb{H})$--orbits in  $\mathbb{CP}^5=\mathbb{P}(\Lambda^2\C^4)$   (Theorem~\ref{thm:pole-orbit-types}) that has an immediate counterpart in the orbit classification of projective dual space $(\mathbb{CP}^5)^\vee$ of hyperplanes of $\mathbb{CP}^5$.\par 
The results of Section~\ref{secHyperPlaneSections} are applied in Section~\ref{secApplications1} to the class of slice--regular polynomials: we give a criterion for planarity (Theorem~\ref{thm:Hyperplane-sectionality-criterion}) and we show that the twistor lift of a given slice--regular function can be contained in several hyperplane sections (Theorem~\ref{thm:second class.}), thus leading, in some cases, non a non--uniqueness of the planar type (Theorem~\ref{thm:unique-planar-type}).\par 
In the final Section~\ref{secApplication2} we study more in depth how the action of $\GL(2,\mathbb{H})$ on $\Gr_2(\C^4)$ is mirrored by a partial action on $\RS$ (Corollary~\ref{cor:partial-action}) and a global action of a smaller group (Theorem~\ref{thm:genuine-action-gamma}). The latter allows to find normal forms of slice--regular polynomials, up to a small residual freedom given by the simultaneous inner conjugation of all intermediate coefficients (Corollary~\ref{cor:residual-action-Nn}).


\medskip

\medskip
\noindent\textbf{Acknowledgments:} C. Liu was supported by   the
		Postdoctoral Fellowship Program and China Postdoctoral
		Science Foundation under Grant Number BX20250067, and the China Postdoctoral Science
		Foundation under Grant Number 2025M773074.
H. Shi was supported by the Zhejiang Provincial Natural Science Foundation of China under Grant No.LQN25A010014 and No.LMS26A010001, the National Natural Science Foundation of China under Grant No. 12671152, and the China Scholarship Council under Grant No. 202509710010.

\section{Preliminaries}\label{secPrel}
In this section, we begin by fixing the coordinate conventions used throughout the paper and by recalling the affine twistor model for slice-regular quaternionic functions. The main purpose of this section is to express the twistor transform as a holomorphic curve in the affine chart \(U_6 \subset \operatorname{Gr}_2(\mathbb C^4)\), and to formulate planarity as containment in a hyperplane section of the Klein quadric.
\subsection{Quaternionic and complex coordinate conventions}
Let $\mb H=\{ a+bi+cj+dk\mid a,b,c,d\in \mb R \}$ be the   algebra of quaternions, and 
let
\[
\mathbb S:=\{q\in \mathbb H\mid q^2=-1\}
\]
be the unit sphere of imaginary units. To any imaginary unit $i\in \mb S$ we associate the  complex plane
\[
\C_i:=\R+i\R\cong \C
\]
and the complex half--plane 
\[
\C_i^{+}:=\{x+yi\in \C_i:\mid y>0\}\subset\C_i\, .
\]
Throughout this paper, we fix two orthogonal imaginary units $i,j\in \mb S $ with
$ij=-ji$: it follows that each quaternion $q=a+bi+cj+dk\in\mb H$ determines uniquely two complex numbers $z=a+bi, w=c+di\in \C_i$, such that 
\begin{equation}\label{eqQatAsComplexPair}
    q=z+w j \, .
\end{equation}
\begin{defn}
The \emph{reflection} $\hat{a}(v)$  of a holomorphic function  $a:V\to \C_i$ on a domain
$V\subset \C_i$ is given by
\begin{equation}\label{eqDefHatA}
    \hat{a}(v):=\overline{a(\overline v)}\, .
\end{equation}
\end{defn}
We recall that $\hat{a}$ is again holomorphic on $V$ whenever $V$ is symmetric with respect to
complex conjugation, or whenever $V\subset \C_i^+$: then $\hat a$ can be regarder as the holomorphic Schwarz reflection of $a$.

For simplicity of notation, we shall omit the subscript \(i\)
and write simply \(\mathbb C\) instead of \(\mathbb C_i\). Thus, whenever complex
matrices such as \(\operatorname{Mat}_{2\times 2}(\mathbb C)\) occur below, the symbol
\(\mathbb C\) should be understood as the fixed slice \(\mathbb C_i\) under the above
identification.

\subsection{Graph coordinates on \(\operatorname{Gr}_2(\mathbb C^4)\)}
In this subsection we recall the Pl\"ucker embedding of $\Gr_2(\C^4)$ in $\mb C\mb P^5$ and describe a standard affine chart in terms of graphs of $2\times 2$ matrices.

The Grassmannian
$$
\Gr_2(\C^4):=\{L\mid L\subset\C^4\textrm{ is a linear subspace and }\dim_{\C}L=2\}
$$
is a complex projective manifold of dimension $4$: its embedding in 
$$\mb \CP^{5}\cong P\left(\bigwedge^2\C^4\right)$$ 
is given by 
the {\em Pl\"ucker embedding}  
\begin{eqnarray}
\Pl:\Gr_2(\C^4)&\hookrightarrow& \mb \mb \CP^{5}\, ,\label{eqPluckEmb}\\
L=\spa\{v_1,v_2\}&\longmapsto& [v_1\wedge v_2]\, ,\nonumber
\end{eqnarray}
 
Indeed, the replacement of $(v_1,v_2)$ by another basis of $L$ corresponds to a rescaling of $v_1\wedge v_2$, which does not affect the projective class thereof.\par 

To the standard   basis $(e_1,e_2,e_3,e_4)$  in $\C^4$ we associate the   basis
\[
E_1=e_3\wedge e_4,\quad
E_2=-e_2\wedge e_4,\quad
E_3=e_2\wedge e_3,\quad
E_4=e_1\wedge e_4,\quad
E_5=-e_1\wedge e_3,\quad
E_6=e_1\wedge e_2
\]
of $\bigwedge^2\C^4$, which in turns determine homogeneous coordinates
\begin{equation}
    \label{eqZetas}[\zeta_1:\zeta_2:\zeta_3:\zeta_4:\zeta_5:\zeta_6]
\end{equation}
on $\CP^{5}$. 
In these   coordinates,  
the image $\Pl(\Gr_2(\C^4))$ of the Pl\"ucker embedding \eqref{eqPluckEmb}
is cut out by the equation
\begin{equation}\label{eq:Klein-eq-unified}
q(\z):=\zeta_{1}\zeta_{6}-\zeta_{2}\zeta_{5}+\zeta_{3}\zeta_{4}=0\, ,
\end{equation}
and it is known as the Klein quadric. 
The coordinates $[\zeta_1:\zeta_2:\zeta_3:\zeta_4:\zeta_5:\zeta_6]$ of an element $\Pl(L)$ will be therefore referred to as the Pl\"ucker coordinates of $L\in\Gr_2(\C^4)$.
\begin{defn}\label{defStdAffCh}
    A \emph{standard affine chart} $\mathcal{U}_6\subset\Gr_2(\C^4)$   consists of all those $L\in \Gr_2(\C^4)$ whose projection onto the first summand of a decomposition
\begin{equation}\label{eqC4C2C2}
    \C^4=\C^2\oplus \C^2
\end{equation}
is an isomorphism. 
\end{defn}
The standard affine chart $\mathcal{U}_6$ can be identified with the 4--dimensional linear space $\Mat_{2\times2}(\C)$ of all $2\times 2$ matrices: if a quadruple $(Z_1,Z_2,Z_3,Z_4)\in\C^4$ of complex numbers is rewritten as a pair $(x,y)$, where $ x=(Z_1,Z_2)\in\C^2$ and $y=(Z_3,Z_4)\in\C^2$, then the graph of the linear map $\Phi:\C^2\to\C^2$, corresponding to the matrix   
\[
\Phi=\begin{pmatrix} a & b \\ c & d \end{pmatrix}\in \Mat_{2\times2}(\C)
\]
is given by
\begin{equation}\label{eqDefEllPhi}
    L(\Phi):=\{(x,y)\mid y=\Phi (x)\}
=\{(Z_1,Z_2,Z_3,Z_4)\in\C^4\mid  Z_3=aZ_1+bZ_2,\    Z_4=cZ_1+dZ_2\}\, .
\end{equation}

Thanks to Lemma~\ref{lem:graph-chart} below, we can  identify
\begin{equation}\label{eq:U_6}
    \mathcal{U}_6\equiv\{[\zeta_1:\zeta_2:\zeta_3:\zeta_4:\zeta_5:\zeta_6]\in \Pl(\Gr_2(\C^4))\mid \zeta_6\neq 0\}\, ,
\end{equation}
which incidentally  clarifies the choice of index $6$ in the symbol  $\mathcal{U}_6$.
\begin{lem}\label{lem:graph-chart}
    For every matrix
\[
\Phi=\begin{pmatrix} a & b \\ c & d\end{pmatrix}\in \Mat_{2\times 2}(\C) 
\]
the graph $L(\Phi)$ is an element of $\mathcal{U}_{6}$, whose  Pl\"ucker coordinates are:
\begin{equation}
\label{eq:graph-plucker}
\Pl(L(\Phi))=[\det\Phi:c:-a:d:-b:1]\, .
\end{equation}
Conversely, every point of $\mathcal{U}_{6}$ has the form $L(\Phi)$ for a unique matrix
$\Phi\in \Mat_{2\times 2}(\C)$.
\end{lem}
\begin{proof}
If
\[
v_1=e_1+ae_3+ce_4,\qquad v_2=e_2+be_3+de_4,
\]
then $L(\Phi)=\spa\{v_1,v_2\}$ and
\[
\begin{aligned}
v_1\wedge v_2
&=(e_1+ae_3+ce_4)\wedge(e_2+be_4+de_4)\\
&=(ad-bc)E_1+cE_2-aE_3+dE_4-bE_5+E_6\, .
\end{aligned}
\]
Hence
\[
\Pl(L(\Phi))=[\det\Phi:c:-a:d:-b:1]\, .
\]
In particular, $L(\Phi)\in \mathcal{U}_{6}$ for every $\Phi\in \Mat_{2\times2}(\C)$.\par 
Conversely,  for every  $L\in \mathcal{U}_{6}$,   there exists a unique $\Phi$, such that $L=L(\Phi)$: indeed, $\Phi=p_2\circ(p_1|_L)^{-1}$, where $p_i$ denote the projection onto the $i$--th factor of~\eqref{eqC4C2C2}, for $i=1,2$.
\end{proof}
We recall that  $\Gr_2(\C^4)$ is a homogeneous manifold of the group $\GL(4,\C)$. Even though  $\mathcal{U}_6$ is \textit{not} preserved by the $\GL(4,\C)$--action, it is nonetheless convenient to show how a transformation $T\in \GL(4,\C)$, given in the block form
\begin{equation}
\label{eq:block-T}
T=\begin{pmatrix} A & B \\ C & D\end{pmatrix}\, ,
\qquad A,B,C,D\in \Mat_{2\times 2}(\C)\, ,
\end{equation}
corresponding to the    decomposition~\eqref{eqC4C2C2} above, acts on an element $L(\Phi)\in\mathcal{U}_6$. 
 
\begin{lem}\label{lem:fractional-linear-graphs}
    Let $T\in \GL(4,\C)$ be as in \eqref{eq:block-T}, and let $\Phi\in \Mat_{2\times 2}(\C)$.
Then the following statements are equivalent:
\begin{enumerate}[label=\textup{(\roman*)}]
\item $T\cdot L(\Phi)\in \mathcal{U}_{6}$;
\item the matrix $A+B\Phi$ is invertible.
\end{enumerate}
Whenever these conditions hold, one has
\begin{equation}
\label{eq:fractional-linear}
T\cdot L(\Phi)=L\bigl((C+D\Phi)(A+B\Phi)^{-1}\bigr)\, .
\end{equation}
\end{lem}
    \begin{proof}
By its definition~\eqref{eqDefEllPhi}, the  plane $L(\Phi)$ consists of pairs $(x,\Phi x)$, with  $x\in \C^{2}$, so that  
\[
T(x,\Phi x)=\bigl((A+B\Phi)x,(C+D\Phi)x\bigr).
\]
Thus the first projection $T\cdot L(\Phi)\to \C^{2}$ is represented by $A+B\Phi$.
Therefore $T\cdot L(\Phi)$ belongs to the affine chart $\mathcal{U}_{6}$ if and only if this first
projection is invertible, i.e. if and only if $A+B\Phi$ is invertible. In that case we can
solve $x=(A+B\Phi)^{-1}y$, so that
\[
T(x,\Phi x)=\bigl(y,(C+D\Phi)(A+B\Phi)^{-1}y\bigr)\, .
\]
Hence $T\cdot L(\Phi)$ is the graph of $(C+D\Phi)(A+B\Phi)^{-1}$, which  proves
\eqref{eq:fractional-linear}.
\end{proof}

\subsection{Quaternionic linear algebra in complex coordinates}
In this subsection we recall the standard complex realization of quaternionic matrices, which allows  to embed the group  $\GL(2,\mb H)$ into  $\GL(4,\C)$.

The matrix
\[
J:=\begin{pmatrix}0&-1\\1&0\end{pmatrix}\in \Mat_{2\times 2}(\C)
\]
allows to define  an anti-linear involution \begin{eqnarray}\label{eq:Theta-constant-new}
\p:\Mat_{2\times 2}(\C) &\longrightarrow &\Mat_{2\times 2}(\C)\, \\
M&\longmapsto&\p(M):=J\,\overline{M}\,J^{-1}\, .\nonumber 
\end{eqnarray}
The fixed point set of $\p$ will be denoted by:
\begin{equation}\label{eq:rho-fixed-new}
\Fix(\p):=\{M\in \Mat_{2\times 2}(\C)\mid \p(M)=M\}.
\end{equation}
\begin{rem}
We use the same symbol $\p$ for the induced anti-linear involution on holomorphic matrix-valued functions on a symmetric domain $D\subset \C$:
\begin{equation}\label{eq:Theta-function-new}
(\p X)(v):=J\,\hat{X}\,J^{-1}.
\end{equation}
{Here, consistently with the notation introduced in \eqref{eqDefHatA}, the hat denotes the  reflection,
\[
\widehat X(v):=\overline{X(\overline v)},
\]
with complex conjugation applied entrywise.}
\end{rem}
Representation~\eqref{eqQatAsComplexPair} of quaternions allow to define the map 
\begin{equation}\label{eq:rho-quaternion}
\r:\mb H\to \Mat_{2\times 2}(\C),
\qquad
\r(z+w j):=\begin{pmatrix}z&-\overline w\\ w&\overline z\end{pmatrix}.
\end{equation}
A direct computation shows that
\begin{equation}\label{eq:H is fixed point}
    \Fix(\p)=\r(\mb H).
\end{equation}

We will denote by the same symbol $\r$ the $\R$--algebra monomorphism
\begin{equation}\label{eq:rho-entrywise}
\r:\Mat_{2\times 2}(\mb H)\to \Mat_{4\times 4}(\C)
\end{equation}
obtained by   applying $\r$ entry-wise on a $2\times 2$ quaternionic matrix. 
By identifying $\Mat_{2\times 2}(\mb H)$ with its image under \eqref{eq:rho-entrywise}, the group of invertible $2\times 2$ quaternionic matrices becomes a subgroup
\begin{equation}\label{eq:GL2H-image}
\GL(2,\mb H)=
\left\{
\begin{pmatrix}A&B\\ C&D\end{pmatrix}\in \GL(4,\C)\mid 
A,B,C,D\in \r(\mb H)
\right\}
\end{equation}
of the group $\GL(4,\C)$ of invertible $4\times 4$ complex matrices.

We observe that 
\begin{equation}\label{eq:rho-invertibles}
\r(\mb H)\cap \GL(2,\C)=\left\{\begin{pmatrix}z&-\overline w\\ w&\overline z,\end{pmatrix}\mid |z|^2+|w|^2\neq 0\right\}=\r(\mb H^{\times})
\end{equation}
is a copy of the multiplicative subgroup
$$
\mb H^\times:=\mb H\smallsetminus\{0\} =\{q\in \mb H\mid |q|^2=|z|^2+|w|^2\neq 0\}
$$
of $\mb H$: it follows that  the center of $\GL(2,\mb H)$ is
\[
\{rI_{4}\mid r\in \R^{\times}\}\, , 
\]
since the center of $\mb H$ is $\R$. 
Accordingly, we set
\begin{equation}\label{eq:PGL2H-def}
\mr{PGL}(2,\mb H):=\GL(2,\mb H)/\R^{\times}.
\end{equation}

\subsection{The affine twistor transform of a slice-regular function}
Recall that a domain \(\Omega\subset \mathbb H\) is called a \emph{slice domain} if
\(\Omega\cap \R\neq \emptyset\) and
\[
\Omega_I:=\Omega\cap \C_I
\]
is a domain in \(\C_I\) for every \(I\in \mathbb S\). It is called \emph{symmetric} if 
\[
x+yI\in \Omega \quad \Longrightarrow \quad x+yJ\in \Omega
\]
for every
\(x,y\in \R\) and every \(I,J\in \mathbb S\).

Let us set 
\[
\Omega_i:=\Omega\cap \C,
\qquad
V:=\Omega\cap \C^{+}.
\]
\begin{defn}
Let \(\Omega\subset \mathbb H\) be a symmetric slice domain. A function
\(f:\Omega\to \mathbb H\) is called \emph{slice regular} if for every \(i\in \mathbb S\),
the restriction
\[
f_i:=f|_{\Omega_i}:\Omega_i\to \mathbb H
\]
is holomorphic with respect to the complex structure given by left multiplication by \(i\), i.e.
\[
\overline{\partial}_{i}f_i(x+yi)
:=
\frac12\Bigl(\frac{\partial}{\partial x}+i\frac{\partial}{\partial y}\Bigr)f_i(x+yi)
=0
\]
on \(\Omega_i\). Denote by $\RS$ the set of all slice regular functions from $\O$ to $\mb H$.
\end{defn}

The splitting lemma \cite[Lemma~2.16 ]{CFGbook18} allows to rewrite the slice $f_i$ of a function \(f\in \RS\) as
\begin{equation}\label{eq:splitting1}
f_i(z)=g(z)+h(z)j\, ,\qquad z\in \Omega_i\, ,
\end{equation}
where 
\begin{equation}\label{eq:splitting}
    g,h:\Omega_i\to \C
\end{equation}
are  holomorphic
functions.
\begin{defn}\label{dfnSplittData}
    The pair $(g,h)$ above constitutes  the \emph{splitting data} of $f\in\RS$.
\end{defn}
Moreover, in view of definition~\eqref{eqDefHatA},    both \(\hat g\) and \(\hat h\)
turn out to be holomorphic on \(\Omega_i\), since   \(\Omega_i\) is symmetric with respect to complex conjugation: this allows to introduce the matrix-valued holomorphic map
\begin{equation}
\label{eq:Phi-f}
\Phi_f(z):=
\begin{pmatrix}
g(z) & -\hat h(z)\\
h(z) & \hat g(z)
\end{pmatrix}\, ,
\qquad z\in \Omega_i\, ,
\end{equation}
which, once restricted  to \(V\), defines the  holomorphic curve
\begin{eqnarray}\label{eq:Gf-chart-new}
    G_f:V &\to& \Gr_{2}(\C^{4})\, ,\\
 v&\mapsto &G_f(v):=L(\Phi_f(v))\nonumber
\end{eqnarray}
associated with $f$. 
\begin{defn}\label{defGraphMatrix}
    The holomorphic map~\eqref{eq:Phi-f} will be referred to as the \emph{graph matrix} of $f$.
\end{defn}
The holomorphic curve $G_f$ is   precisely   the twistor transform of \(f\) in the affine chart \(\mathcal{U}_6\): indeed, $G_f(v)$ has   Pl\"ucker coordinates
\begin{equation}
\label{eq:Plucker-Gf}
G_f(v)
=
[g(v)\hat g(v)+h(v)\hat h(v):h(v):-g(v):\hat g(v):\hat h(v):1]
\end{equation}
(see  Lemma~\ref{lem:graph-chart}) and obviously   \(G_f(V)\subset \mathcal{U}_6\), see \cite[Theorem~5.7]{GSS14}.\par 

We shall use the antiholomorphic map
\begin{equation}
\label{eq:j-on-C4}
\mathbf{j}(Z_{1},Z_{2},Z_{3},Z_{4})
=
(-\overline{Z}_{2},\overline{Z}_{1},-\overline{Z}_{4},\overline{Z}_{3}),
\end{equation}
 satisfying \(\mathbf{j}^{2}=-\id\), which  induces the real structure
\begin{equation}
\label{eq:sigma}
\sigma[\zeta_{1}:\zeta_{2}:\zeta_{3}:\zeta_{4}:\zeta_{5}:\zeta_{6}]
=
[\overline{\zeta}_{1}:\overline{\zeta}_{5}:-\overline{\zeta}_{4}:-\overline{\zeta}_{3}:\overline{\zeta}_{2}:\overline{\zeta}_{6}]
\end{equation}
on \(\mb P(\bigwedge^{2}\C^{4})\).\par 
The following theorem is the form of the twistor correspondence that we shall use.

\begin{thm}\label{thm:GSS-correspondence}
Let
$
G:V\to\Gr_2(\C^4)
$
be a holomorphic curve. Then the following are equivalent.

\begin{enumerate}
\item[\textup{(i)}] There exists a function \(f\in\RS\) such that
$
G=G_f.
$

\item[\textup{(ii)}] The curve \(G\) admits a holomorphic extension
$
\widetilde G:\Omega_i\to\Gr_2(\C^4)
$
such that
 $
\widetilde G(\Omega_i)\subset \mathcal U_6,
$
and
\[
\widetilde G(\overline z)=\sigma(\widetilde G(z))
\qquad
\text{for every }z\in\Omega_i.
\]
\end{enumerate}
Moreover, when these conditions hold, the function \(f\) is unique, and the
extension in \(\textup{(ii)}\) is necessarily
\[
\widetilde G=L\circ\Phi_f.
\]
\end{thm}
\begin{proof}
If \(f\in\RS\), then the construction above gives the holomorphic map
\[
\widetilde G_f=L\circ\Phi_f:\Omega_i\to\Gr_2(\C^4),
\]
with image in the affine chart \(\mathcal U_6\), and
\[
G_f=\widetilde G_f|_V.
\]
The reality condition
\[
\widetilde G_f(\overline z)=\sigma(\widetilde G_f(z))
\]
is exactly the reality property of the twistor transform in the affine chart.
Thus \(\textup{(i)}\Rightarrow\textup{(ii)}\).

Conversely, \(\textup{(ii)}\Rightarrow\textup{(i)}\), together with the
uniqueness of \(f\), is precisely the affine-chart form of
\cite[Theorem~5.7]{GSS14}.  The same theorem also identifies the extension
with \(\widetilde G_f\).
\end{proof}

\subsection{Planarity and hyperplane type}
A hyperplane section $X_z$ of $\Gr_{2}(\C^{4})$ is the intersection 
\[
X_{z}:=\Gr_{2}(\C^{4})\cap \Pi_{z}
\]
of $\Gr_{2}(\C^{4})$ with the hyperplane 
\[
\Pi_{z}:=\left\{[\zeta]\in \mb P\left(\bigwedge^{2}\C^{4}\right)\mid  z_{1}\zeta_{1}+\cdots +z_{6}\zeta_{6}=0\right\}
\]
determined by a nonzero covector $z=(z_{1},\dots,z_{6})\in \C^{6}$.  
\begin{defn}
\label{def:hyperplane-sectional}
A holomorphic curve $G:V\to \Gr_{2}(\C^{4})$ is called \emph{planar} if
$G(V)\subset X_{z}$ for some nonzero $z$.
\end{defn}
 
\begin{prop}
\label{prop:hyperplane-relation}
The twistor transform $G_{f}$ of $f\in\RS$ is
planar if and only if there exist constants $z_{1},\dots,z_{6}\in \C$, not all
zero, such that
\begin{equation}
\label{eq:hyperplane-relation}
z_{1}(g\hat g+h\hat h)+z_{2}h-z_{3}g+z_{4}\hat g+z_{5}\hat h+z_{6}=0
\end{equation}
identically on $V$, where $g,h$ are the splitting data of $f$ (Definition~\ref{dfnSplittData}).
\end{prop}

\begin{proof}
By Definition~\ref{def:hyperplane-sectional}, the curve $G_{f}$ is planar  if
and only if there exists a nonzero covector $z=(z_{1},\dots,z_{6})$ such that
\[
z_{1}\zeta_{1}+\cdots +z_{6}\zeta_{6}=0
\]
for every point $[\zeta_{1}(v):\dots :\zeta_{6}(v)]=G_{f}(v)$. Substituting the explicit formula
\eqref{eq:Plucker-Gf} yields exactly \eqref{eq:hyperplane-relation}.
\end{proof}

We shall need the symmetric bilinear form $B(\z,\eta)$ that polarizes \eqref{eq:Klein-eq-unified}, i.e., such that \(q(\z)=\frac12 B(\z,\z)\):
\begin{equation}\label{eq:B-unified}
B(\z,\eta)
:=
\z_{1}\eta_{6}+\z_{6}\eta_{1}-\z_{2}\eta_{5}-\z_{5}\eta_{2}+\z_{3}\eta_{4}+\z_{4}\eta_{3}\, .
\end{equation}
This bilinear form allows us to introduce the polar hyperplane to a point \([z]\in \CP^{5}\).

\begin{defn}
The hyperplane
\begin{equation}\label{eq:polar-hyperplane-unified}
H_{[z]}:=\{[\eta]\in \CP^{5}: B(z,\eta)=0\}.
\end{equation}
is the \emph{polar hyperplane} of $[z]$.
\end{defn}
Let $(\CP^{5})^{\vee}$ denote the projective dual of $\CP^{5}$, i.e., the set of all projective hyperplanes in $\CP^{5}$. The correspondence  
\begin{equation}\label{eq:kappa-unified}
\kappa:\CP^{5}\longrightarrow (\CP^{5})^{\vee},
\qquad
[z_1:z_2:z_3:z_4:z_5:z_6]\longmapsto [z_6:-z_5:z_4:z_3:-z_2:z_1]\, ,
\end{equation}
which is clearly an  isomorphism, is the \emph{polarity map}.\par
\begin{prop}\label{prop:kappa-unified}
The polarity map \(\kappa\) is equivariant with respect to the projective
\(\bigwedge^2\r\)-action of \(\PGL(2,\mb H)\) and 
 
\begin{equation}\label{eq:trans}
\Pi_{\kappa([z])}=H_{[z]}\qquad\text{ for any }[z]\in \CP^{5}\, .
\end{equation}
\end{prop}
\begin{proof}
The polarity map of a smooth
quadric is known to be equivariant: see, for instance, Harris~\cite[Chapter~22]{Harris1992}.  
    The equality \(\Pi_{\kappa([z])}=H_{[z]}\) follows immediately from the coordinate
formula for \(B\): indeed,
\[
B(z,\eta)
=
z_6\eta_1-z_5\eta_2+z_4\eta_3+z_3\eta_4-z_2\eta_5+z_1\eta_6.
\]
Thus the covector defining \(H_{[z]}\) is precisely
\[
[z_6:-z_5:z_4:z_3:-z_2:z_1]=\kappa([z]).
\]
\end{proof}
Let \(A\subset (\CP^{5})^{\vee}\) be a \(G\)-orbit, that is a subset of the form
\begin{equation}
    A:=G\cdot [z]\, ,
\end{equation}
with \([z]\in  (\CP^{5})^{\vee}\).

\begin{defn}\label{def:A-planar-unified}

A holomorphic curve \(G:V\to \Gr_{2}(\C^{4})\), such that
\[
G(V)\subset X_{z}
\qquad \text{for some } [z]\in A\, ,
\]
will be called \emph{planar of type  $A$}.
\end{defn}

\section{Hyperplane sections under the \(\PGL(2,\mathbb H)\)-action}\label{secHyperPlaneSections}
This section classifies the hyperplane sections of the Klein quadric up to the natural \(\PGL(2,\mathbb H)\)-action. Via the polarity induced by the Plücker quadratic form, this amounts to classifying the corresponding orbits of poles in \(\mb P(\Lambda^2\mathbb C^4)\).
\subsection{The real quadratic model}
In this subsection, we 
will work in
\[
\mathbb P(\mathbb V)\cong\CP^5\, ,\quad
\mathbb V:=\mathbb C^6\, ,
\]
equipped with the homogeneous coordinates
\[
[p]=[p_1:p_2:p_3:p_4:p_5:p_6]\in\mathbb P(\mathbb V)\, ,
\]
that correspond to~\eqref{eqZetas} in the case $\mathbb V=\Lambda^2\mathbb{C}^4$: accordingly, the   Klein quadric~\eqref{eq:Klein-eq-unified} 
\[
Q=\{[p]\in\mathbb P(\mathbb V):q(p)=0\}
\]
is the  smooth quadric hypersurface associated with  the quadratic form
\[
q(p)=p_1p_6-p_2p_5+p_3p_4\, ,
\]
whose polarization 
\[
B(p,r):=q(p+r)-q(p)-q(r)
\]
reads
\[
B(p,r)
=
p_1r_6+p_6r_1-p_2r_5-p_5r_2+p_3r_4+p_4r_3\, ,
\]
see also~\eqref{eq:B-unified}.\par  
We will use the non-degenerate bilinear form \(B\)  to identify \(\mathbb P(\mathbb V)\) with the dual
projective space \(\mathbb P(\mathbb V^\vee)\cong(\CP^5)^\vee\).
\begin{defn}
Let \(\mathbb W\subset \mathbb V\) be a linear subspace: the subspace
\[
\mathbb W^\perp:=\{r\in \mathbb V: B(w,r)=0\ \text{for all } w\in \mathbb  W\}
\]
will be referred to as the   
\emph{orthogonal complement} of $\mathbb W$.
\end{defn}
In particular, we let 
\[
p^\perp:=\spa\{p\}^\perp=\{[r]\in\mathbb P(\mathbb V):B(p,r)=0\}
\]
for any point
\([p]\in\mathbb P(\mathbb V)\), thus obtaining a point--to--hyperplane correspondence
\[
[p]\longmapsto p^\perp\, .
\]
\begin{defn}
    We call \([p]\) the \emph{pole} of the hyperplane \(p^\perp\) and \(p^\perp\) the \emph{polar hyperplane} of \([p]\).
\end{defn}
\begin{defn}
    The projectivization $\mathbb{P}(\mathbb W^\perp)$ of the orthogonal complement of  $\mathbb W$ will be referred to as the \emph{polar space} of $\mathbb P(\mathbb W)$.
\end{defn}
It is worth observing that  \(\mathbb P(\mathbb W)^\perp\) is the intersection of all polar hyperplanes
associated with points of \(\mathbb P(W)\):
\[
\mathbb P(\mathbb W)^\perp=\mathbb P(\mathbb W^\perp)
=
\bigcap_{[ w]\in \mathbb P(\mathbb W)} w^\perp .
\]
\begin{defn}
    A nonzero vector \(p\in \mathbb V\) is called \emph{isotropic}, or \emph{null}, with respect to \(q\)
if
\[q(p)=0.\]
\end{defn}
The following are equivalent:
\begin{itemize}
    \item[(i)] $p$ is null;
\item[(ii)] the point \([p]\in \mathbb P(\mathbb V)\) lies on the Klein quadric
\(Q\);
\item[(iii)]
$[p]\in p^\perp$. 
\end{itemize}
\begin{defn}
A nonzero vector \(p\in\mathbb V\) is called \emph{anisotropic} if
\[q(p)\neq 0.
\]
\end{defn}
\begin{defn}
    A linear subspace \(\mathbb W\) of $\mathbb V$ is a \emph{totally isotropic subspace} with respect to \(B\) if
\[
B|_{\mathbb W\times \mathbb W}=0\, .
\]
\end{defn}
Obviously, these facts are equivalent:
\begin{itemize}
    \item[(i)] $\mathbb W$ is totally isotropic;
\item[(ii)] $B(w_1,w_2)=0$
 for all $w_1,w_2\in \mathbb W$;
\item[(iii)] $\mathbb W\subset \mathbb W^\perp$;
\item[(iv)] $\mathbb P(\mathbb W)\subset \mathbb P(\mathbb W)^\perp$;
\item[(v)] every point of \(\mathbb P(\mathbb W)\) lies in the polar hyperplane
of every other point of \(\mathbb P(\mathbb W)\);
\item[(vi)] $q|_\mathbb W=0$;
\item[(vii)] $\mathbb P(\mathbb W)\subset Q$.
\end{itemize}
Since \(B\) is non-degenerate,  one has
\[
\dim \mathbb W+\dim\mathbb W^\perp=\dim \mathbb V
\]
for every linear subspace \(\mathbb W\subset \mathbb V\): in particular, if \(\mathbb W\) is totally isotropic, then \(\mathbb W\subset \mathbb W^\perp\), which implies
\[
\dim \mathbb W\leq \dim \mathbb W^\perp\, ,
\]
and then
\[
2\dim \mathbb W\leq \dim \mathbb V=\dim\mathbb C^6=6\, .
\]
It follows that
\[
\dim \mathbb W\leq 3
\]
for a totally isotropic subspace $\mathbb W\subset \mathbb V$, 
that is, 
\[
\dim L=\dim \mathbb W-1\leq 2\, ,
\]
where \(L=\mathbb P(\mathbb W)\subset \mathbb P(\mathbb V)\cong \CP^5\)
is a projective linear subspace contained in the Klein quadric \(Q\):  the maximal projective linear subspaces contained in the Klein
quadric are projective planes \(\mathbb{CP}^2\).

The anti-linear involution
\[
\sigma:\mathbb V\to \mathbb V
\]
introduced earlier in~\eqref{eq:sigma},  reads
\[
\sigma(p_1,p_2,p_3,p_4,p_5,p_6)
=
(\overline p_1,\overline p_5,-\overline p_4,
-\overline p_3,\overline p_2,\overline p_6)
\]
in the present setting. 
If we denote by
\[
 V_{\mathbb R}:=\operatorname{Fix}(\sigma)
\]
its fixed-point set, then
\[
\mathbb  V=V_{\mathbb R}\otimes_{\mathbb R}\mathbb C.
\]
It is easy to see that \(q|_{V_{\mathbb R}}\) has real signature \((1,5)\):
indeed, a vector  \(p\in V_{\mathbb R}\) has the form
\[
p=(r,c,-b,\overline b,\overline c,s),
\qquad
r,s\in\mathbb R,\quad b,c\in\mathbb C\, ,
\]
and hence
\[
q(p)=rs-|b|^2-|c|^2.
\]
\begin{lem}
Every totally isotropic real subspace of \(V_{\mathbb R}\)
has dimension at most one.
\end{lem}
\begin{proof}
The real quadratic space \((V_{\mathbb R},q)\) of signature \((1,5)\) has Witt
index  one: the result then follows from~\cite[Proposition~8.11]{EKM08}.
\end{proof}
Under the standard identification of quaternionic M\"obius transformations with
the conformal group of \(S^4\), the group
\[
G=\PGL(2,\mathbb H)
\]
acts projectively as the identity component of the real orthogonal group $$O(V_\R,q):=\{A\in \GL_\R(V_\R)\mid A^*(q)=q\}$$ of
\((V_{\mathbb R},q)\). 
{Here \(A^*q\) denotes the pull-back of the quadratic form \(q\) by the
real-linear automorphism \(A\). More explicitly,
\[
(A^*q)(v):=q(Av),\qquad v\in V_{\mathbb R}.
\]
Thus the condition \(A^*q=q\) means that \(A\) preserves the quadratic form
\(q\), i.e.
\[
q(Av)=q(v),\qquad \forall\, v\in V_{\mathbb R}.
\]
Equivalently, if \(B\) is the polar bilinear form associated with \(q\), then
\[
B(Av,Aw)=B(v,w),\qquad \forall\, v,w\in V_{\mathbb R}.
\]}
Every vector \(p\in\mathbb V\) can be written uniquely as
\[
p=x+iy,
\qquad
x,y\in V_{\mathbb R},
\]
where
\[
x=\frac{p+\sigma p}{2},
\qquad
y=\frac{p-\sigma p}{2i}.
\]

Multiplying \(p=x+iy\) by a non-zero complex scalar
\(\lambda=a+ib\) changes the pair \((x,y)\) to
\[
(x,y)\mapsto (ax-by,\ bx+ay).
\]
Thus projectivizing \(p\) allows rotations and common rescalings of the ordered
pair \((x,y)\).

For \(p=x+iy\), define the real Gram matrix
\[
M_p:=
\begin{pmatrix}
B(x,x)&B(x,y)\\
B(x,y)&B(y,y)
\end{pmatrix}.
\]
This is the Gram matrix of the ordered pair \((x,y)\) with respect to the real
bilinear form \(B|_{V_{\mathbb R}}\).

Since \(\sigma p=x-iy\), we have
\[
B(p,\sigma p)
=
B(x+iy,x-iy)
=
B(x,x)+B(y,y).
\]
Therefore
\[
B(p,\sigma p)=\operatorname{tr}M_p.
\]

Similarly,
\[
B(p,p)
=
B(x,x)-B(y,y)+2iB(x,y).
\]
Hence
\[
|B(p,p)|
=
\sqrt{(B(x,x)-B(y,y))^2+4B(x,y)^2}.
\]

For \(q(p)\neq0\), equivalently \(B(p,p)\neq0\), define
\[
\tau([p])
:=
\frac{B(p,\sigma p)}{|B(p,p)|}
=
\frac{B(p,\sigma p)}{2|q(p)|}.
\]
In terms of \(M_p\), this is
\[
\tau([p])
=
\frac{\operatorname{tr}M_p}
{\sqrt{(M_{11}-M_{22})^2+4M_{12}^2}}.
\]
Thus \(\tau\) is the normalized trace of the real Gram matrix of \(x,y\).

The quantity \(\tau\) is well-defined on projective points.  Indeed, if
\(p\) is replaced by \(\lambda p\),  {where
\(\lambda\in\mathbb C^\times\),} then
\[
B(\lambda p,\sigma(\lambda p))
=
|\lambda|^2B(p,\sigma p),
\]
whereas
\[
|B(\lambda p,\lambda p)|
=
|\lambda|^2|B(p,p)|.
\]
Thus the ratio is unchanged.

\subsection{Orbit classification in \(\mb P(\Lambda^2\mathbb C^4)\)}

We introduce below  a list of special points of $\mathbb V$.
\begin{defn}
A $\PGL(2,\mathbb{H})$--orbit $\PGL(2,\mathbb{H})\cdot [p]\subset\mathbb{P}  \mathbb V$ will be called:
\begin{enumerate}[label=\textup{(A\arabic*)}]
\item \textbf{real null}, if $p$ equals
\[
p_N:=[0:0:0:0:0:1];
\]
\item \textbf{non-real isotropic}, if $p$ equals
\[
p_{Q\smallsetminus N}:=[-1:0:-i:i:0:1];
\]
\end{enumerate}
\begin{enumerate}[label=\textup{(B\arabic*)}]
\item \textbf{real time-like}, if $p$ equals
\[
p_+:=[1:0:0:0:0:1];
\]
\item \textbf{real space-like}, if $p$ equals
\[
p_-:=[1:0:0:0:0:-1];
\]
\item \textbf{non-real degenerate}, if $p$ equals
\[
p_{\mathrm{deg}}:=[0:0:-i:i:0:1];
\]
\item \textbf{non-real Lorentzian }, if $p$ equals
\[
p_\theta=[e^{i\theta}:0:0:0:0:1]\, ,\quad 0<\theta<\pi\, ;
\]
\item \textbf{non-real negative-definite},  if $p$ equals
\[
p_\lambda:=[1:0:-i\lambda:i\lambda:0:-1]\, ,\quad 0<\lambda<1.
\]
\end{enumerate}
\end{defn}
\begin{defn}
    We say that $p\in \mathbb V$ (or $[p]\in\mathbb{P}\mathbb V$) has \emph{type} $p_0$ (or \emph{type} $A$) if $[p]\in A:=\PGL(2,\mathbb{H})\cdot[p_0]$.
\end{defn}
We prove now that the types of all the elements of $\mathbb V$ appears exaclty once in the above list.
\begin{thm}\label{thm:pole-orbit-types}
Each $\PGL(2,\mathbb{H})$--orbit  in $\mathbb{P}(\mathbb V)$ appears exactly once in the above list.
\end{thm}

\begin{proof}
(A1) Assume that \(q(p)=0\).  Since
\[
B(p,p)=2q(p),
\]
this is equivalent to
\[
B(p,p)=0.
\]
Using the formula
\[
B(p,p)
=
B(x,x)-B(y,y)+2iB(x,y),
\]
we obtain
\[
B(x,x)=B(y,y),
\qquad
B(x,y)=0.
\]
Thus \(M_p\) is a scalar matrix.

If this scalar is zero, then \(q(x)=q(y)=0\) and \(x\) and \(y\) are  orthogonal.
Because a real quadratic space of signature \((1,5)\) has Witt index one,
a totally isotropic subspace has dimension at most one.  Hence \(x\) and
\(y\) are real-linearly dependent.  Therefore \([p]\) has a real null
representative.  This gives the orbit represented by
\[
p_N=[0:0:0:0:0:1].
\]

(A2) If the scalar is negative, then \(x\) and \(y\) span a negative definite
two-plane, and they are orthogonal with equal negative length.  This gives
the non-real isotropic orbit.  A representative is
\[
p_{Q\smallsetminus N}=[-1:0:-i:i:0:1].
\]
Indeed,
\[
q(p_{Q\smallsetminus N})
=
(-1)\cdot 1+(-i)i
=
-1+1
=
0,
\]
and
\[
[\sigma p_{Q\smallsetminus N}]\neq[p_{Q\smallsetminus N}].
\]

To see that this representative gives the whole orbit, write
\[
p_0=(-1,0,-i,i,0,1).
\]
Then
\[
x_0=\frac{p_0+\sigma p_0}{2}=(-1,0,0,0,0,1),
\qquad
y_0=\frac{p_0-\sigma p_0}{2i}=(0,0,-1,1,0,0),
\]
and
\[
B(x_0,x_0)=B(y_0,y_0)=-2,
\qquad
B(x_0,y_0)=0.
\]
After multiplying \(p\) by a positive real scalar, we may also assume
\[
B(x,x)=B(y,y)=-2,
\qquad
B(x,y)=0.
\]
Hence the map \(x\mapsto x_0,\ y\mapsto y_0\) is an isometry between the
two negative definite planes.  By Witt's extension theorem \cite[Chapter~III, Theorem~3.9]{Artin1957}, this isometry
extends to an isometry of \((V_{\mathbb R},q)\).  Therefore \([p]\) lies in
the \(G\)-orbit of \(p_{Q\smallsetminus N}\).

Finally, the scalar cannot be positive.  Indeed, it would make
\[
\operatorname{span}_{\mathbb R}\{x,y\}
\]
a positive definite two-plane.  This is impossible because
\((V_{\mathbb R},q)\) has signature \((1,5)\), and hence has positive index
one.

\medskip

(B) Now assume that \(q(p)\neq0\).  If
\[
[\sigma p]=[p],
\]
then \([p]\) has a real representative.  A real non-isotropic line satisfies either
$q(p)>0$ or $q(p)<0$.\par 
(B1-B2) These two possibilities are represented by
\[
p_+=[1:0:0:0:0:1],
\qquad
p_-=[1:0:0:0:0:-1].
\]
They satisfy
\[
q(p_+)=1,
\qquad
q(p_-)=-1.
\]
Moreover,
\[
\tau(p_+)=1,
\qquad
\tau(p_-)=-1.
\]

It remains to consider the case
\[
q(p)\neq0,
\qquad
[\sigma p]\neq[p].
\]
Then \(x\) and \(y\) are real-linearly independent, and
\[
E_p:=\operatorname{span}_{\mathbb R}\{x,y\}
\]
is a genuine two-dimensional real plane in \(V_{\mathbb R}\).
Since \(V_{\mathbb R}\) has signature \((1,5)\), the restriction of \(q\) to
\(E_p\) is either Lorentzian $(1,1)$, degenerate $(0,1,1)$, or negative definite $(0,2)$.  It cannot be
positive definite.

The range of \(\tau\) is determined by the identity
\[
(\operatorname{tr}M_p)^2
-
\bigl((M_{11}-M_{22})^2+4M_{12}^2\bigr)
=
4\det M_p.
\]
Thus,
\[
\det M_p<0
\quad\Longleftrightarrow\quad
-1<\tau<1,
\]
which is the Lorentzian case;
\[
\det M_p=0
\quad\Longleftrightarrow\quad
\tau=-1
\]
in the non-real degenerate case; and
\[
\det M_p>0,\quad \operatorname{tr}M_p<0
\quad\Longleftrightarrow\quad
\tau<-1,
\]
which is the negative definite case.\par

(B3) For the degenerate case, take
\[
p_{\mathrm{deg}}=[0:0:-i:i:0:1].
\]
Then
\[
q(p_{\mathrm{deg}})
=
(-i)i
=
1,
\]
and a direct computation gives
\[
\tau(p_{\mathrm{deg}})=-1.
\]
The real and imaginary parts span a degenerate real plane.\par
(B4) For the Lorentzian case, take
\[
p_\theta=[e^{i\theta}:0:0:0:0:1],
\qquad
0<\theta<\pi.
\]
Then
\[
q(p_\theta)=e^{i\theta}.
\]
Also,
\[
\sigma p_\theta=[e^{-i\theta}:0:0:0:0:1],
\]
and hence
\[
B(p_\theta,\sigma p_\theta)
=
e^{i\theta}+e^{-i\theta}
=
2\cos\theta.
\]
Since
\[
|B(p_\theta,p_\theta)|
=
|2q(p_\theta)|
=
2,
\]
we get
\[
\tau(p_\theta)=\cos\theta.
\]
As \(0<\theta<\pi\), this gives precisely
\[
-1<\tau<1.
\]\par 
(B5) For the negative definite case, take
\[
p_\lambda=[1:0:-i\lambda:i\lambda:0:-1],
\qquad
0<\lambda<1.
\]
Then
\[
q(p_\lambda)
=
1\cdot(-1)+(-i\lambda)(i\lambda)
=
-1+\lambda^2
=
-(1-\lambda^2).
\]
Moreover,
\[
B(p_\lambda,\sigma p_\lambda)
=
-2(1+\lambda^2),
\]
whereas
\[
|B(p_\lambda,p_\lambda)|
=
2(1-\lambda^2).
\]
Therefore
\[
\tau(p_\lambda)
=
-\frac{1+\lambda^2}{1-\lambda^2}
<
-1.
\]
As \(0<\lambda<1\), this realizes all values \(\tau<-1\).

Finally, Witt's extension theorem \cite[Chapter~III, Theorem~3.9]{Artin1957} and Sylvester's law of inertia \cite[Chapter~I, Proposition 3.2 (3)]{Lam2005} imply that these normal forms exhaust the
orbits.  Indeed, any isometry between two subspaces of a non-degenerate
quadratic space extends to an isometry of the whole space.  Therefore two points \([p]\) and \([p']\) with
the same normal form for the real Gram matrix of their real and imaginary
parts lie in the same orbit.  The preceding list gives all possible normal
forms in signature \((1,5)\).  Hence there are no further orbits.
\end{proof}

\subsection{Dual hyperplane types}
The corresponding classification of hyperplanes, that are points of \(\mb P(\mb V^\vee)=(\CP^5)^\vee\), is
obtained via the isomorphism \(\kappa\) given earlier in~\eqref{eq:kappa-unified}.
\begin{defn}
\label{def:pole-types}
A hyperplane  \([z]\in(\CP^5)^\vee\) has type $p_0$ if such is the type of its pole
\[
[p_z]=\kappa^{-1}([z])\, .
\]
\end{defn}
Thanks to Theorem~\ref{thm:pole-orbit-types}, the complete list of hyperplanes types is as follows:
\begin{enumerate}[label=\textup{(\alph*)}]
\item[(A1)] \emph{\(N\)-tangent type}:
\[
q(p_z)=0,
\qquad
[\sigma p_z]=[p_z].
\]
Equivalently, the pole lies on
\[
N:=Q\cap \mb P(V_{\R})\cong \HP^1.
\]

\item[(A2)] \emph{\(Q\smallsetminus N\)-tangent type}:
\[
q(p_z)=0,
\qquad
[\sigma p_z]\neq[p_z].
\]

\item[(B1)] \emph{real positive type}:
\([p_z]\) is projectively real and, for a real representative
\(u\in V_{\R}\) of \([p_z]\),
\[
q(u)>0.
\]
This is often called real time-like type.

\item[(B2)] \emph{real negative type}:
\([p_z]\) is projectively real and, for a real representative
\(u\in V_{\R}\) of \([p_z]\),
\[
q(u)<0.
\]
This is often called real space-like type.

\item[(B3)] \emph{non-real Lorentzian type}:
\[
q(p_z)\neq0,\qquad [\sigma p_z]\neq[p_z],
\qquad -1<\tau([p_z])<1.
\]

\item[(B4)] \emph{non-real degenerate type}:
\[
q(p_z)\neq0,\qquad [\sigma p_z]\neq[p_z],
\qquad \tau([p_z])=-1.
\]

\item[(B5)] \emph{non-real negative-definite type}:
\[
q(p_z)\neq0,\qquad [\sigma p_z]\neq[p_z],
\qquad \tau([p_z])<-1.
\]
\end{enumerate}

For convenience, we provide
the standard pole representatives and their corresponding dual covectors are
as follows:
\[
\begin{array}{l|l|l}
\text{type} & \text{pole representative }p & \text{dual covector representative }z=\kappa(p)\\
\hline
N\text{-tangent}
&
[0:0:0:0:0:1]
&
[1:0:0:0:0:0]
\\[1mm]
Q\smallsetminus N\text{-tangent}
&
[-1:0:-i:i:0:1]
&
[1:0:i:-i:0:-1]
\\[1mm]
\text{real positive}
&
[1:0:0:0:0:1]
&
[1:0:0:0:0:1]
\\[1mm]
\text{real negative}
&
[1:0:0:0:0:-1]
&
[1:0:0:0:0:-1]
\\[1mm]
\text{non-real degenerate}
&
[0:0:-i:i:0:1]
&
[1:0:i:-i:0:0]
\\[1mm]
\text{non-real Lorentzian}
&
[e^{i\theta}:0:0:0:0:1]
&
[1:0:0:0:0:e^{i\theta}]
\\[1mm]
\text{non-real negative-definite}
&
[-1:0:-i\lambda:i\lambda:0:1]
&
[1:0:i\lambda:-i\lambda:0:-1]
\end{array}
\]
where \(0<\theta<\pi\) and \(0<\lambda<1\).

\section{Planar twistor transforms of slice-regular polynomials}\label{secApplications1}
The theory of slice regular functions of one quaternionic variable originates from C.~G.~Cullen's notion of regularity\cite{C65} and then
was developed into a modern function theory over the quaternions\cite{GGSD06,GGS07}. A fundamental feature of this framework is that it naturally contains quaternionic polynomials and power series. Indeed, G.~Gentili and D.~C.~Struppa \cite{GGSD06} explicitly showed that the basic monomials \(q^{n}a\) are regular, and hence every polynomial
\[
f(q)=\sum_{m=0}^{n} q^{m}a_{m}, \qquad a_{m}\in\mathbb H,
\]
with coefficients on the right is slice regular. In this section, we  apply the preceding projective-geometric framework to slice-regular polynomials. 
\subsection{A coefficient-space criterion for planarity}

\begin{lem}\label{lem:real-linear}
If \(z=(z_1,\ldots,z_6)\in\mathbb C^6\), then 
\[\Lambda_{z}(b+cj)
:=
z_2 c-z_3 b+z_4\overline b+z_5\overline c,
\]
where $b,c\in\mathbb C$, 
defines an $\mathbb{R}$-linear map $\Lambda_z:\mathbb{H}\to\mathbb C$. Conversely, for every $\mathbb{R}$-linear map $\Lambda:\mathbb{H}\to\mathbb C $ there exists $z\in\mathbb C^6$, such that $\Lambda=\Lambda_z$.
\end{lem}
\begin{proof}
The first statement is immediate. For the converse, write
\[
q=x+yi+uj+vk=(x+yi)+(u+vi)j,
\]
so that $b=x+yi$ and $c=u+vi$. Then
\[
\Lambda(q)=(-z_3+z_4)x-i(z_3+z_4)y+(z_2+z_5)u+i(z_2-z_5)v.
\]
Now let $\Lambda:\mathbb{H}\to\mathbb C$ be $\mathbb{R}$-linear. Then there exist unique complex numbers $A,B,C,D\in\mathbb C$ such that
\[
\Lambda(x+yi+uj+vk)=Ax+By+Cu+Dv.
\]
If we set
\[
z_2:=\frac{C-iD}{2},\qquad z_5:=\frac{C+iD}{2},
\qquad
z_3:=\frac{-A+iB}{2},\qquad z_4:=\frac{A+iB}{2},
\]
then the preceding formula yields $\Lambda=\Lambda_z$. Uniqueness follows because the coefficients of $x,y,u,v$ determine $z_2,z_3,z_4,z_5$ uniquely.
\end{proof}

\begin{rem}\label{rem:middle-coordinates}
The notation \(\Lambda_z\) deliberately uses a six-tuple \(z\in\mathbb C^6\), although the map only depends on the middle four coordinates
\((z_2,z_3,z_4,z_5)\). This convention is chosen to match the Pl\"ucker
hyperplane equation later, where the same six-tuple \(z\) will be used as a
covector defining a hyperplane in \(\mathbb P(\Lambda^2\mathbb C^4)\).
\end{rem}
\begin{thm}\label{thm:Hyperplane-sectionality-criterion}
Let
\[
f(q)=\sum_{m=0}^n q^m a_m
\]
be a non-constant slice-regular polynomial, and set
\[
S_f:=\spa_{\R}\{a_1,\dots,a_n\}\subset \mb H.
\]
Then the following conditions are equivalent:
\begin{enumerate}
\renewcommand{\labelenumi}{\rm(\theenumi)}
    \item the twistor transform $G_f$ is planar;
    
    \item there exists a non-zero $\R$-linear map $\Lambda:\mb H\to\C$, such that
    \[
    \Lambda(a_m)=0\qquad \forall m=1,\dots,n;
    \]
    
    \item the non-constant coefficients of $f$ are contained in a proper real linear subspace of $\mb H$, that is 
    \[
    \dim_{\R}S_f\leq 3.
    \]
\end{enumerate}
\end{thm}

\begin{proof}
{We prove \((1)\Rightarrow(2)\Rightarrow(1)\): the equivalence   \((2)\Leftrightarrow(3)\) is an elementary fact of linear algebra.\par 

Proof of $(1)\Rightarrow(2)$. If $G_f$ is planar, then by Proposition \ref{prop:hyperplane-relation},
there exists a non-zero vector
$z=(z_1,\dots,z_6)\in\C^6$
such that
\begin{equation}\label{eq:hyperplane-relation1}
z_1P(v)+z_2h(v)-z_3g(v)+z_4\widehat g(v)+z_5\widehat h(v)+z_6=0
\qquad \text{for all }v\in V.
\end{equation}
 Decompose each coefficient \(a_m\) for all \(m=0,1,\dots,n,\) as
\begin{equation*}
a_m=b_m+c_mj,\qquad b_m,c_m\in\C.
\end{equation*}
Here\begin{equation}\label{eq:gv-hv}
g(v)=\sum_{m=0}^n b_m v^m,\qquad
h(v)=\sum_{m=0}^n c_m v^m,\qquad
\widehat g(v)=\sum_{m=0}^n \overline{b_m}\,v^m,\qquad
\widehat h(v)=\sum_{m=0}^n \overline{c_m}\,v^m.
\end{equation}

Since $\deg(f)=n$, at least one of $b_n$, $c_n$ is non-zero: hence
\[
P(v)=g(v)\widehat g(v)+h(v)\widehat h(v)
\]
has degree $2n$, with leading coefficient
\[
|b_n|^2+|c_n|^2=|a_n|^2\neq 0.
\]
On the other hand,   $\deg(g)$, $\deg(h)$, $\deg(\widehat g)$, and $\deg(\widehat h)$ are all $\leq n$: therefore the term $z_1P(v)$ is the only possible source of a $v^{2n}$-term in \eqref{eq:hyperplane-relation1}, and we must have
\[
z_1=0.
\]
Consequently, \eqref{eq:hyperplane-relation1} reduces to
\begin{equation}\label{eq:reduced-hyperplane-relation}
z_2h(v)-z_3g(v)+z_4\widehat g(v)+z_5\widehat h(v)+z_6=0
\qquad \text{for all }v\in V.
\end{equation}
We obtain
\begin{equation}\label{eq:reduced-hyperplane-relation1}
\sum_{m=0}^n
\bigl(z_2c_m-z_3b_m+z_4\overline{b_m}+z_5\overline{c_m}\bigr)v^m
+z_6=0.
\end{equation}

We define $\Lambda:=\Lambda_z$ as in Lemma~\ref{lem:real-linear}:
\[
\Lambda(b+cj):=z_2c-z_3b+z_4\overline b+z_5\overline c.
\]
Comparing coefficients of $v^m$ in \eqref{eq:reduced-hyperplane-relation1} gives
\[
\Lambda(a_m)=0\qquad \forall m=1,\dots,n,
\]
and also
\[
\Lambda(a_0)+z_6=0.
\]
It remains to check that $\Lambda\neq 0$. Towards a contradiction, if $\Lambda=0$, then, by the uniqueness statement in Lemma~\ref{lem:real-linear}, we would have
\[
z_2=z_3=z_4=z_5=0 
\]
and  \eqref{eq:reduced-hyperplane-relation} would reduce to $z_6=0$: together with $z_1=0$, this would yield
\[
z_1=z_2=z_3=z_4=z_5=z_6=0,
\]
that is, $z=0$, contrary to assumption.\par  

Proof of $(2)\Rightarrow(1)$. 
If  $\Lambda:\mb H\to\C$ is a non-zero $\R$-linear map with
\[
\Lambda(a_m)=0\qquad \forall m\geq 1\, .
\]
Write
\[
a_m=b_m+c_mj,\qquad b_m,c_m\in\mathbb C.
\]
then, by Lemma~\ref{lem:real-linear}, there exist uniquely determined complex numbers $z_2,z_3,z_4,z_5$ such that
\[
\Lambda(b_m+c_mj)=z_2c_m-z_3b_m+z_4\overline b_m+z_5\overline c_m.
\]
Since \(\Lambda\neq 0\), the tuple
\((z_2,z_3,z_4,z_5)\) is not zero.

If we set
\[
z_1:=0,\qquad z_6:=-\Lambda(a_0)
\]
and set
\[
z=(z_1,z_2,z_3,z_4,z_5,z_6)\in\mathbb C^6.
\]
Then \(z\neq 0\).

Recall that
\[
G_f(v)
=
\bigl[P(v):h(v):-g(v):\widehat g(v):\widehat h(v):1\bigr],
\]Then
\[
\begin{aligned}
\ell_z(G_f(v))
&:=
z_1P(v)+z_2h(v)-z_3g(v)
+z_4\widehat g(v)+z_5\widehat h(v)+z_6  \\
&=
z_2h(v)-z_3g(v)
+z_4\widehat g(v)+z_5\widehat h(v)-\Lambda(a_0).
\end{aligned}
\]
Based on \eqref{eq:gv-hv}, we have
\[
\ell_z(G_f(v))
=
\sum_{m=0}^n
\bigl(
z_2c_m-z_3b_m+z_4\overline{b_m}
+z_5\overline{c_m}
\bigr)v^m -\Lambda(a_0) 
=
\sum_{m=1}^n \Lambda(a_m)v^m.
\]
 Then 
\[
\ell_z(G_f(v))=0\qquad \forall v\in V,
\]
i.e.,  $G_f(\C)$ is contained in the hyperplane section $X_z$.
Therefore \(G_f\) is planar.\par  }

Proof of \((2)\Leftrightarrow(3)\). 
If
\[
S_f:=\spa_{\R}\{a_1,\dots,a_n\}\subset \mb H\simeq \R^4,
\]
then $S_f$ is proper if and only if its annihilator
\[
S_f^\circ:=\{\ell\in \operatorname{Hom}_{\R}(\mb H,\R)\mid \ell|_{S_f}=0\}
\]
is non-zero; equivalently, there exists a non-zero $\R$-linear functional vanishing on $S_f$. This follows from the dimension formula
\[
\dim_{\R} S_f^\circ = 4-\dim_{\R} S_f.
\]
Hence \((3)\) is equivalent to the existence of a non-zero $\R$-linear map
\[
\Lambda:\mb H\to\C
\]
such that
\[
\Lambda(a_m)=0\qquad \forall m\ge1.
\]
\end{proof}
We show now that the planarity of the twistor curve associated with $f$  is independent of the constant term $a_0$: in other words, translating $f$ by $a_0$ does not affect the planarity of $f$. 
\begin{cor}\label{prop:constant-term-planarity}
Let
\[
f(q)=a_0+\widetilde f(q),
\qquad
\widetilde f(q)=\sum_{m=1}^n q^m a_m .
\]
Then there exists a fixed matrix
\[
M_{a_0}:=
\begin{pmatrix}
I_2 & 0\\
\rho(a_0) & I_2
\end{pmatrix}
\in \GL(4,\mathbb C)
\]
such that
\[
G_f=M_{a_0}\cdot G_{\widetilde f}.
\]
In particular, \(G_f\) is planar if and only if \(G_{\widetilde f}\) is planar.
\end{cor}

\begin{proof}
By the polynomial expression for the graph matrix,
\[
\Phi_f(v)
=
\sum_{m=0}^n v^m\rho(a_m)
=
\rho(a_0)+\sum_{m=1}^n v^m\rho(a_m)
=
\rho(a_0)+\Phi_{\widetilde f}(v).
\]
Fix \(v\in V\). Since
\[
G_{\widetilde f}(v)=L(\Phi_{\widetilde f}(v)),
\]
based on \eqref{eqDefEllPhi}, every vector in \(G_{\widetilde f}(v)\) has the form
\[
\begin{pmatrix}
x\\
\Phi_{\widetilde f}(v)x
\end{pmatrix},
\qquad x\in\mathbb C^2.
\]
Applying \(M_{a_0}\) gives
\[
M_{a_0}
\begin{pmatrix}
x\\
\Phi_{\widetilde f}(v)x
\end{pmatrix}
=
\begin{pmatrix}
I_2 & 0\\
\rho(a_0) & I_2
\end{pmatrix}
\begin{pmatrix}
x\\
\Phi_{\widetilde f}(v)x
\end{pmatrix}
=
\begin{pmatrix}
x\\
\bigl(\rho(a_0)+\Phi_{\widetilde f}(v)\bigr)x
\end{pmatrix}
=
\begin{pmatrix}
x\\
\Phi_f(v)x
\end{pmatrix}.
\]
Hence
\[
M_{a_0}\cdot G_{\widetilde f}(v)\subseteq G_f(v).
\]
Both sides are two-dimensional subspaces of \(\mathbb C^4\), so the inclusion is an equality:
\[
M_{a_0}\cdot G_{\widetilde f}(v)=G_f(v).
\]
Since this holds for every \(v\in V\), we obtain the equality of curves
\[
G_f=M_{a_0}\cdot G_{\widetilde f}.
\]

Finally, since \(M_{a_0}\in \GL(4,\mathbb C)\), its induced projective action sends projective hyperplanes to projective hyperplanes. Therefore \(G_f\) is contained in a hyperplane section if and only if \(G_{\widetilde f}\) is contained in a hyperplane section. Thus planarity is independent of the constant term.
\end{proof}

\begin{cor}\label{cor:easy-consequences}
Let $f(q)=\sum_{m=0}^n q^m a_m$ be a slice-regular polynomial.
\begin{enumerate}
\renewcommand{\labelenumi}{\rm(\theenumi)}
    \item If $\deg(f)\leq 3$, then $f$ is planar.
    \item If
    \[
f(q)=a_0+qa_1+q^2a_2+q^3a_3+q^4a_4
    \]
    is a quartic polynomial, then $f$
    is planar  if and only if the four coefficients $a_1,a_2,a_3,a_4$ are linearly dependent over $\R$.
\end{enumerate}
\end{cor}

\begin{proof}
Part \((1)\) follows because at most three vectors cannot span the real four-dimensional space $\mb H$. Part \((2)\) is the specialization of the same criterion to $N=4$.
\end{proof}

\subsection{Hyperplanes containing a polynomial twistor curve}
In this subsection, we further classify the planar polynomials in terms of Definition \ref{def:A-planar-unified}.

Let
\[
f(q)=\sum_{m=0}^n q^m a_m
\]
be a non-constant slice-regular polynomial, and set
\[
S_f:=\spa_{\R}\{a_1,\dots,a_n\}\subset \mb H.
\]

Throughout this paper, we always assume that $f$ is planar.
By Theorem \ref{thm:Hyperplane-sectionality-criterion}, we have \[
r:=\dim_{\R}S_f\le3.
\]

For $\Lambda\in \Hom_{\R}(\mb H,\C)$, we denote
\[\rank_\R\Lambda:=\dim_\R\Lambda(\mb H).\]

Define
\[
\mathcal A_f
:=
\{\Lambda\in\Hom_{\R}(\mb H,\C):\Lambda|_{S_f}=0\}.
\]

Take \(0\neq\Lambda\in\mathcal A_f\) and write
\[
\Lambda(b+cj)
=
z_2c-z_3b+z_4\overline b+z_5\overline c.
\]
Set the covector
\begin{equation}\label{eq:z-Lambda-f}
z_{\Lambda,f}:=(0,z_2,z_3,z_4,z_5,-\Lambda(a_0))\in \mb V^\vee.
\end{equation}
In this case, the pole of \(\Pi_{z_{\Lambda,f}}\) is
\begin{equation}\label{eq:p-Lambda-f}
p_{\Lambda,f}:=\kappa^{-1}([z_{\Lambda,f}])
=
[-\Lambda(a_0):-z_5:z_4:z_3:-z_2:0].
\end{equation}

Define
\[
\Delta(\Lambda):=z_3z_4-z_2z_5,
\qquad
\nu(\Lambda):=|z_2|^2+|z_3|^2+|z_4|^2+|z_5|^2.
\]
A direct calculation gives
\begin{equation}\label{eq:1}
    q(p_{\Lambda,f})=\Delta(\Lambda),
\end{equation}
and 
\begin{equation}\label{eq:2}
 B(p_{\Lambda,f},\sigma p_{\Lambda,f})=-\nu(\Lambda).
\end{equation}

We now state the classification of planar polynomial twistor curves.  We shall
see that, among all dual hyperplane orbits listed in
Definition~\ref{def:pole-types}, only three types can occur for hyperplane
sections containing a non-constant polynomial twistor curve.

Let
\[
\mathcal O_-:=G\cdot[1:0:0:0:0:-1]\subset\CP^5
\]
be the real space-like pole orbit.  Let
\[
\mathcal O_{Q\smallsetminus N}:=G\cdot[-1:0:-i:i:0:1]\subset\CP^5
\]
be the non-real isotropic pole orbit.  Finally, let
\[
\mathcal O_{\mathrm{nd}}
:=
\bigcup_{0<\lambda<1}
G\cdot[-1:0:-i\lambda:i\lambda:0:1]
\subset\CP^5
\]
be the non-real negative-definite pole family.

For convenience, we denote the corresponding dual hyperplane orbits by
\[
A_-:=\kappa(\mathcal O_-),
\qquad
A_{Q\smallsetminus N}:=\kappa(\mathcal O_{Q\smallsetminus N}),
\qquad
A_{\mathrm{nd}}:=\kappa(\mathcal O_{\mathrm{nd}})
\subset(\CP^5)^\vee.
\]
Thus \(A_-\), \(A_{Q\smallsetminus N}\), and \(A_{\mathrm{nd}}\) are respectively
the dual real space-like orbit, the dual \(Q\smallsetminus N\)-tangent orbit, and
the dual non-real negative-definite family.

\begin{thm}\label{thm:second class.}
 The hyperplane sections containing \(G_f\) are parametrized by
\[
\mb P(\mathcal A_f)\cong\CP^{3-r}.
\]
More precisely, for any \(0\neq\Lambda\in\mathcal A_f\), one has
\[
G_f\subset X_{z_{\Lambda,f}},
\]
and every hyperplane section containing \(G_f\) is obtained in this way, up to
multiplying \(\Lambda\) by a non-zero complex scalar.

Moreover, for a non-zero \(\Lambda\in\mathcal A_f\), the corresponding
hyperplane section \(X_{z_{\Lambda,f}}\) has the following type:
\begin{enumerate}[label=\textup{(\roman*)}]
\item If
\[
\rank_{\R}\Lambda=1,
\]
then \(f\) is \(A_-\)-planar.

\item If
\[
\rank_{\R}\Lambda=2
\qquad\text{and}\qquad
\Delta(\Lambda)=0,
\]
then \(f\) is \(A_{Q\smallsetminus N}\)-planar.

\item If
\[
\rank_{\R}\Lambda=2
\qquad\text{and}\qquad
\Delta(\Lambda)\neq0,
\]
then \(f\) is \(A_{\mathrm{nd}}\)-planar.
\end{enumerate}
\end{thm}

\begin{proof}
    The first statement is from the proof of Theorem \ref{thm:Hyperplane-sectionality-criterion}. Furthermore, let \(s_1,\dots,s_r\) be a real basis of \(S_f\).  The evaluation map
\[
\Hom_{\R}(\mb H,\C)\to \C^r,
\qquad
\Lambda\mapsto(\Lambda(s_1),\dots,\Lambda(s_r)),
\]
is surjective: extend \(s_1,\dots,s_r\) to a real basis of \(\mb H\) and
prescribe arbitrary complex values on the first \(r\) basis vectors.
Since \(\Hom_{\R}(\mb H,\C)\) is a complex vector space of dimension \(4\),
we get
\[
\dim_{\C}\mathcal A_f=4-r.
\]
Thus
\[
\mb P(\mathcal A_f)\cong\CP^{3-r}.
\]

 It remains to relate the real rank of \(\Lambda\) to these pole conditions.
Let \(L_\Lambda\) be the real \(2\times4\) matrix of the real-linear map
\[
\Lambda:\mb H\simeq\R^4\to\C\simeq\R^2.
\]
From
\[
\Lambda(x+yi+uj+vk)
=
(-z_3+z_4)x-i(z_3+z_4)y+(z_2+z_5)u+i(z_2-z_5)v,
\]
one obtains the identity
\begin{equation}\label{eq:rank-identity}
\det(L_\Lambda L_\Lambda^T)
=
\nu(\Lambda)^2-4|\Delta(\Lambda)|^2.
\end{equation}
Since \(L_\Lambda L_\Lambda^T\) is positive semidefinite and
\(\Lambda\neq0\), the real rank of \(\Lambda\) is \(1\) if and only if
\begin{equation}\label{eq:2026572008}
\nu(\Lambda)=2|\Delta(\Lambda)|>0.    
\end{equation}

Suppose first that \(\rank_{\R}\Lambda=1\).  The image of \(\Lambda\) is a
real line in \(\C\).  Hence there exist \(\mu\in\C^{\times }\) and a real-valued
linear functional \(\lambda:\mb H\to\R\) such that
\[
\Lambda=\mu\lambda.
\]
The covector \(z_{\Lambda,f}\), and hence also the pole \(p_{\Lambda,f}\),
is a complex multiple of the covector and pole obtained from \(\lambda\). As \(\lambda\) is real-valued, the pole \(p_{\lambda,f}\) is fixed by
the real structure \(\sigma\), that is,
\[
p_{\lambda,f}\in V_{\mathbb R}.
\]
Therefore
\[
[p_{\Lambda,f}]=[p_{\lambda,f}]
\]
is projectively real.
  Moreover, by \eqref{eq:2} and \eqref{eq:2026572008}, one has
\[
B(p_{\Lambda,f},\sigma p_{\Lambda,f})=-\nu(\Lambda)<0.
\]
On the other hand, since \(p_{\Lambda,f}=\mu p_{\lambda,f}\), then
\[
B(p_{\Lambda,f},\sigma p_{\Lambda,f})
=
|\mu|^2B(p_{\lambda,f},p_{\lambda,f})
=
2|\mu|^2q(p_{\lambda,f}).
\]
Thus \(q(p_{\lambda,f})<0\).  This is real negative type, and so $f$ is $A_-$-planar.

Now suppose that \(\rank_{\R}\Lambda=2\) and \(\Delta(\Lambda)=0\).  Then by \eqref{eq:1},
\[
q(p_{\Lambda,f})=0.
\]
Also \eqref{eq:2} gives
\[
B(p_{\Lambda,f},\sigma p_{\Lambda,f})=-\nu(\Lambda)<0.
\]
If \([p_{\Lambda,f}]\) were projectively real, then we could write
\([p_{\Lambda,f}]=[u]\) with \(u\in V_{\R}\).  Since \(q(u)=0\), we would
have
\[
B(u,u)=2q(u)=0.
\]
This would force
\[
0>-\nu(\Lambda)=B(p_{\Lambda,f},\sigma p_{\Lambda,f})=0.
\]
a contradiction.  Hence the pole is isotropic and not
projectively real.  This is the \(Q\smallsetminus N\)-tangent type, and so $f$ is $A_{Q\smallsetminus N}$-planar.

Finally suppose that \(\rank_{\R}\Lambda=2\) and \(\Delta(\Lambda)\neq0\).
Then \(q(p_{\Lambda,f})\neq0\).  Since the rank is \(2\), the determinant in
\eqref{eq:rank-identity} is positive, and therefore
\[
\nu(\Lambda)>2|\Delta(\Lambda)|.
\]
Hence \eqref{eq:1} and \eqref{eq:2} give
\[
\tau([p_{\Lambda,f}])
=
\frac{B(p_{\Lambda,f},\sigma p_{\Lambda,f})}
{2|q(p_{\Lambda,f})|}
=
-\frac{\nu(\Lambda)}{2|\Delta(\Lambda)|}
<-1.
\]
 Therefore the type is non-real negative-definite, and so $f$ is $A_{\mr{nd}}$-planar.
\end{proof}
\subsection{Uniqueness and non-uniqueness of planar type}
The refined classification should be understood as a classification of the
hyperplane sections containing \(G_f\), or equivalently of the pairs
\[
(G_f,X_{z_{\Lambda,f}}),
\qquad
0\neq\Lambda\in\mathcal A_f.
\]
It does not assign a unique orbit type to the polynomial \(f\) itself. More generally, we have the following.

\begin{thm}\label{thm:unique-planar-type}
 The following hold.

\begin{enumerate}[label=\textup{(\roman*)}]
\item If
$
r=3,
$
then \(G_f\) is contained in a unique hyperplane section.  This unique
hyperplane section is of type \(A_-\).  Hence \(f\) has a unique planar type.

\item If
$
r\le2,
$
then \(G_f\) is contained in more than one hyperplane section.  In fact, the
same polynomial \(f\) is simultaneously
\[
A_-\text{-planar},\qquad
A_{Q\smallsetminus N}\text{-planar},\qquad
A_{\mathrm{nd}}\text{-planar}.
\]
In particular, \(f\) has no unique planar type.
\end{enumerate}
\end{thm}
\begin{proof}
If \(r=3\), then by Theorem \ref{thm:second class.},
\[
\mathbb P(\mathcal A_f)\cong\CP^0,
\]
so there is exactly one hyperplane section containing \(G_f\).  Moreover,
every non-zero \(\Lambda\in\mathcal A_f\) vanishes on the real \(3\)-plane
\(S_f\).  Since \(\Lambda\neq0\), its real rank is \(1\).  By the refined
classification, the unique hyperplane section is of type \(A_-\).

\medskip

Now assume \(r\le2\).  Then $
\mathbb P(\mathcal A_f)$
has positive dimension.  Hence \(G_f\) is contained in more than one
hyperplane section.

We now show that the three possible types all occur.  Since
\[
\dim_{\mathbb R}(\mathbb H/S_f)=4-r\ge2,
\]
there exists a non-zero real-linear functional
\[
\ell:\mathbb H/S_f\to\mathbb R.
\]
Let
\[
\pi:\mathbb H\to \mathbb H/S_f
\]
be the quotient map and set
\[
\lambda:=\ell\circ\pi:\mathbb H\to\mathbb R.
\]
Then \(\lambda\) is non-zero and
\[
\lambda|_{S_f}=0.
\]
Viewing \(\mathbb R\) as the real axis in \(\mathbb C\), we may regard
\(\lambda\) as a map from $\mathbb H$ to $\mathbb C.$
Then
\[
\Lambda\in\mathcal A_f.
\]
Moreover, since \(\Lambda\neq0\) and
\[
\Lambda(\mathbb H)\subset \mathbb R\subset\mathbb C,
\]
the image of \(\Lambda\) is a one-dimensional real subspace of \(\mathbb C\).
Thus
\[
\rank_{\mathbb R}\Lambda=1,
\]
and thus \(f\) is \(A_-\)-planar.

Next consider the quadratic expression
\[
\Delta(\Lambda)=z_3z_4-z_2z_5
\]
on \(\mathcal A_f\).  We claim that \(\Delta\) is not identically zero on \(\mathcal A_f\).
Indeed, we have proven that there exists
$
\Lambda\in\mathcal A_f
$
with
$
\rank_{\mathbb R}\Lambda=1.
$
For this \(\Lambda\), one has
\[
\det(L_\Lambda L_\Lambda^T)=0.
\]
By \eqref{eq:rank-identity},
\[
\nu(\Lambda)=2|\Delta(\Lambda)|.
\]
Since \(\Lambda\neq0\), we have
\[
\nu(\Lambda)>0,
\]
and so
\[
\Delta(\Lambda)\neq0.
\]

As
$
\dim_{\mathbb C}\mathcal A_f\ge2,
$
and \(\mathbb C\) is algebraically closed, the non-zero homogeneous
quadratic polynomial
$
\Delta|_{\mathcal A_f}
$
has a non-trivial zero.  Choose
$
0\neq\Lambda_0\in\mathcal A_f
$
with
$
\Delta(\Lambda_0)=0.
$
Such a \(\Lambda_0\) cannot have real rank \(1\), because rank \(1\) implies
\(\Delta\neq0\).  Hence
\[
\rank_{\mathbb R}\Lambda_0=2.
\]
Therefore \(f\) is \(A_{Q\smallsetminus N}\)-planar.

Finally, rank \(2\) elements with \(\Delta\neq0\) also exist.  Indeed, note that \[
\rank_{\mathbb R}\Lambda=2
\quad\Longleftrightarrow\quad
\det(L_\Lambda L_\Lambda^T)>0.
\]
Since \(\Lambda\mapsto \det(L_\Lambda L_\Lambda^T)\) is continuous, the subset of $\mc A_f$ consisting of elements with
rank-two is nonempty and open in the real topology of \(\mathcal A_f\). Then the polynomial \(\Delta|_{\mathcal A_f}\) vanishes on a non-empty open
subset of the real vector space underlying \(\mathcal A_f\).  

If the polynomial \(\Delta|_{\mathcal A_f}\) vanishes on a non-empty open
subset of the real vector space underlying \(\mathcal A_f\), then it is identically zero. However, we have proven that $\Delta$ is not identically zero. Therefore,
 we can choose
$
\Lambda_1\in\mathcal A_f
$
such that
\[
\rank_{\mathbb R}\Lambda_1=2,
\qquad
\Delta(\Lambda_1)\neq0.
\]
By the refined classification, \(f\) is \(A_{\mathrm{nd}}\)-planar.
\end{proof}

\section{Projective transformations and orbit equivalence}\label{secApplication2}
The preceding sections describe planar twistor curves through their containing hyperplanes. We now study how quaternionic projective transformations act on the underlying slice-regular functions and how this action leads to orbit equivalence for polynomials.

\subsection{The partial \(\GL(2,\mathbb H)\)-action}
The full \(\GL(4,\mathbb C)\)-action on \(\operatorname{Gr}_2(\mathbb C^4)\) is too large to preserve the real structure underlying slice-regular functions. The relevant transformations are those coming from \(\GL(2,\mathbb H)\), but their action on functions is only partially defined because the affine chart \(U_6\) need not be preserved.

 Under the identification \(\mb H^2\cong \C^4\), the group of automorphisms of $\mathbb{H}^2$ can be identified with the subgroup
\begin{equation}\label{eq:GL(2,H)}
    \GL(2,\mb H)=\{T\in \GL(4,\C)\mid T\circ \mathbf{j}=\mathbf{j}\circ T\}
\end{equation}
of the group of automorphisms of $\mathbb{C}^4$ that commute with   antiholomorphic map  $\mathbf{j}$, defined as in \eqref{eq:j-on-C4}.
Indeed, if we write
\[
T=\begin{pmatrix}A&B\\ C&D\end{pmatrix},
\qquad A,B,C,D\in \Mat_{2\times 2}(\C),
\]
then the relation \(T\circ \mathbf{j}=\mathbf{j}\circ T\) is equivalent to
\[
A,B,C,D\in\Fix(\p)=\rho(\mb H)\, ,
\]
and then \eqref{eq:GL(2,H)} follows from~\eqref{eq:GL2H-image}.

\begin{lem}
\label{lem:commutes-sigma}
For every $T\in \GL(2,\mb H)$, the induced projective action of $T$ on
$\mb P(\bigwedge^{2}\C^{4})$ preserves $\Gr_{2}(\C^{4})$ and commutes with the real
structure $\sigma$.
\end{lem}
\begin{proof}
The subset $\Gr_{2}(\C^{4})$   of $\mb P(\bigwedge^{2}\C^{4})$ is $\GL(4,\mathbb{C})$--invariant (being made of projective classes of decomposable 2-vectors) and then, in particular, it is preserved by $T$.\par

By the definition of $\sigma$ we have 
\[
\sigma([u\wedge v])=[\mathbf{j}u\wedge \mathbf{j}v]
\]
for any $[u\wedge v]\in \mb P(\bigwedge^{2}\C^{4})$: from the fact that $T$ commutes with $\mathbf{j}$, i.e.,   $T\mathbf{j}=\mathbf{j}T$, it follows  that 
\[
T\cdot \sigma([u\wedge v])
=[T\mathbf{j}u\wedge T\mathbf{j}v]
=[\mathbf{j}Tu\wedge \mathbf{j}Tv]
=\sigma(T\cdot [u\wedge v])\, ,
\]
i.e., $T$ commutes with $\sigma$ on $\mb P(\bigwedge^{2}\C^{4})$ as well.
\end{proof}

Thanks to Theorem \ref{thm:GSS-correspondence},  if a holomorphic curve \(G:V\to \Gr_2(\C^4)\) is the twistor transform of a regular slice function, then its image must be contained in $\mathcal{U}_6$, cf.   \eqref{eq:U_6}. 
\begin{defn}
\label{def:admissible}
Let $f\in\RS$, and let $G_{f}:V\to \Gr_{2}(\C^{4})$ be its twistor
transform. A transformation  $T\in \GL(2,\mb H)$ is called $f$--\emph{admissible} if
\[
T\cdot G_{f}(V)\subset \mathcal{U}_{6}\, .
\]
\end{defn}
\begin{rem}\label{rem:admissible for f graph-chart}
    If $T\in \GL(2,\mb H)$ is given in the form  \eqref{eq:block-T}, then $T$ is $f$--admissible  if
and only if
\begin{equation}
\label{eq:admissible-condition}
\det(A+B\Phi_{f}(v))\neq 0\qquad \forall v\in V.
\end{equation}
\end{rem}

\begin{thm}
\label{thm:conservative-action}
Let  
$G_{f}:V\to \Gr_{2}(\C^{4})$ be the twistor transform of $f\in \RS$, and let $T\in \GL(2,\mb H)$ be an $f$--admissible transformation: then there exists a unique
slice-regular function
\[
T^{*}f\in \RS
\]
such that
\begin{equation}
\label{eq:transform-equality}
G_{T^{*}f}=T\cdot G_{f}.
\end{equation}
Moreover, 
\begin{equation}
\label{eq:Phi-transform}
\Phi_{T^{*}f}(v)=(C+D\Phi_{f}(v))(A+B\Phi_{f}(v))^{-1},\qquad v\in V\, ,
\end{equation}
where $\Phi_{T^{*}f}$ is  the graph matrix of $T^*f$ (Definition~\ref{defGraphMatrix}), and $A$, $B$, $C$ and $D$ are given by the block form~\eqref{eq:block-T} of $T$.
\end{thm}

\begin{proof}
Set
\[
\Psi(v):=(C+D\Phi_{f}(v))(A+B\Phi_{f}(v))^{-1},\qquad v\in V.
\]
Since $\Phi_{f}$ is holomorphic on $V$ and $\det(A+B\Phi_{f}(v))\neq 0$ for all $v\in V$,
the map $\Psi:V\to \Mat_{2\times 2}(\C)$ is holomorphic.\par 

Moreover,
\[
T\cdot G_{f}(v)=T\cdot L(\Phi_{f}(v))=L(\Psi(v))
\]
for each $v\in V$ thanks to Lemma~\ref{lem:fractional-linear-graphs}, so that  $T\cdot G_{f}$ is a holomorphic curve with values in $\mathcal{U}_{6}$.

Next, $T$ commutes
with the real structure $\sigma$ (Lemma~\ref{lem:commutes-sigma}) and $\widetilde G_{f}$ satisfies the reality condition of
Theorem~\ref{thm:GSS-correspondence}: it follows that
\[
(T\cdot\widetilde G_{f})(\overline v)
=T\cdot\widetilde G_{f}(\overline v)
=T\cdot \sigma(\widetilde G_{f}(v))
=\sigma(T\cdot \widetilde G_{f}(v)).
\]
Therefore Theorem~\ref{thm:GSS-correspondence} applies and yields a unique $T^{*}f\in \RS$ whose twistor transform is $T\cdot G_{f}$, and 
\eqref{eq:transform-equality} is fulfilled.

Finally, both $G_{T^{*}f}(v)$ and $T\cdot G_{f}(v)$ lie in the affine chart $\mathcal{U}_{6}$ and
coincide there: by the uniqueness statement in Lemma~\ref{lem:graph-chart}, their graph
matrices must coincide, i.e.,  
\[
\Phi_{T^{*}f}(v)=\Psi(v)=(C+D\Phi_{f}(v))(A+B\Phi_{f}(v))^{-1}\, ,
\]
which is exactly \eqref{eq:Phi-transform}.
\end{proof}

\begin{cor}
\label{cor:partial-action}
The correspondence $(T,f)\mapsto T^{*}f$ defines a partial left action of $\GL(2,\mb H)$ on the
class of slice-regular functions on symmetric slice domains: whenever both sides are
well-defined, one has
\[
(ST)^{*}f=S^{*}(T^{*}f),\qquad I^{*}f=f.
\]
\end{cor}

\begin{proof}
By Theorem~\ref{thm:conservative-action},
\[
G_{(ST)^{*}f}=ST\cdot G_{f}=S\cdot (T\cdot G_{f})=S\cdot G_{T^{*}f}=G_{S^{*}(T^{*}f)}.
\]
The uniqueness statement in Theorem~\ref{thm:conservative-action} implies that the two
slice-regular functions coincide. The identity statement is immediate.
\end{proof}
We now provide a simple example showing that $T^*$ maps constants to constants. \par 

\begin{exam}\label{ex:constant-mobius}
Let $T\in \GL(2,\mb H)$ be given by 
$T=\begin{pmatrix}\rho(\a)&\rho(\beta)\\ \rho(\gamma)&\rho(\delta)\end{pmatrix} $
and let $f_a\equiv a$ be the constant function.
If $T$ is $f_a$-admissible, then
\[
T^*f_a \equiv (\gamma+\delta a)(\alpha+\beta a)^{-1}.
\]
\end{exam}

\begin{proof} 
Since $\Phi_{f_a}\equiv \rho(a)$, formula~\eqref{eq:Phi-transform} becomes
\[
\Phi_{T^*f_a}
=
\bigl(\rho(\gamma)+\rho(\delta)\rho(a)\bigr)
\bigl(\rho(\alpha)+\rho(\beta)\rho(a)\bigr)^{-1}.
\]
Since $\rho$ is a $\R$-algebra monomorphism, it follows that
\[
\rho(\gamma+\delta a)\,\rho(\alpha+\beta a)^{-1}
=
\rho\!\left((\gamma+\delta a)(\alpha+\beta a)^{-1}\right).
\]
Hence $T^*f_a$ is the   function taking the constant  value
\[
(\gamma+\delta a)(\alpha+\beta a)^{-1}.
\qedhere
\]
\end{proof}
Passing to polynomials of degree $n>0$, we show below that such polynomials need to to be preserved by the partial action of $\GL(2,\mathbb{H})$. To this end, let
\[
f(q)=\sum_{m=0}^n q^m a_m,
\qquad a_m\in \mb H,
\]
be a slice-regular polynomial, where
\[
a_m=b_m+c_mj,
\qquad b_m,c_m\in \C\, ,
\]
and let
\[T=\begin{pmatrix}\rho(\a)&\rho(\beta)\\ \rho(\gamma)&\rho(\delta)\end{pmatrix} \in \GL(2,\mb H)\]
be an $f$--admissible transformation.
\begin{prop}\label{prop:polynomial-phi-formula}
The graph matrix of $f$ is given by
\begin{equation}\label{eq:phi-polynomial-expansion}
\Phi_f(v)=\sum_{m=0}^n v^m\,\r(a_m),
\qquad v\in V\, ,
\end{equation}
 and 
\begin{equation}\label{eq:phi-polynomial-transform}
\Phi_{T^*f}(v)
=
\Bigl(\rho(\g)+\rho(\d)\sum_{m=0}^n v^m\r(a_m)\Bigr)
\Bigl(\rho(\a)+\rho(\b)\sum_{m=0}^n v^m\r(a_m)\Bigr)^{-1}.
\end{equation}
In particular, for a general admissible $T$, the function $T^*f$ need not be a polynomial.
\end{prop}

\begin{proof}
Since
\[
f(v)=g(v)+h(v)j,
\qquad
 g(v)=\sum_{m=0}^n b_m v^m,
\qquad
 h(v)=\sum_{m=0}^n c_m v^m,
\]
we have
\[
\widehat g(v)=\overline{g(\overline v)}=\sum_{m=0}^n \overline{b_m}\,v^m,
\qquad
\widehat h(v)=\overline{h(\overline v)}=\sum_{m=0}^n \overline{c_m}\,v^m.
\]
Therefore
\[
\Phi_f(v)
=
\begin{pmatrix}
\sum_{m=0}^n b_m v^m & -\sum_{m=0}^n \overline{c_m} v^m\\
\sum_{m=0}^n c_m v^m & \sum_{m=0}^n \overline{b_m} v^m
\end{pmatrix}
=
\sum_{m=0}^n v^m
\begin{pmatrix}
 b_m & -\overline{c_m}\\
 c_m & \overline{b_m}
\end{pmatrix}.
\]
By the definition of $\r$, the last matrix is exactly $\r(a_m)$, hence
\eqref{eq:phi-polynomial-expansion} follows.
Substituting \eqref{eq:phi-polynomial-expansion} into
\eqref{eq:Phi-transform} gives \eqref{eq:phi-polynomial-transform}.
The final assertion is immediate from the fact that the right-hand side of
\eqref{eq:phi-polynomial-transform} is, in general, a rational function in $v$.
\end{proof}

\begin{prop}\label{prop:hyperplane-invariant}
Let \(A\subset(\CP^5)^\vee\) be a \(\PGL(2,\mathbb H)\)-orbit of
hyperplanes. Let \(f\in\RS\), and let \(T\in\GL(2,\mathbb H)\) be
\(f\)-admissible. If \(G_f\) is planar of type \(A\), then
\(G_{T^*f}\) is also planar of type \(A\). In particular, planarity is
preserved by every admissible transformation.
\end{prop}
\begin{proof}
Let
$
\mathcal T:=\bigwedge^2\rho(T)
$
be the induced linear action on \(\bigwedge^2\C^4\). Its projectivization acts
on
\[
\mathbb P(\bigwedge^2\C^4)\cong\CP^5
\]
and preserves the Grassmannian.

We denote by
$
T^\vee
$
the induced dual action on hyperplanes. Thus, for a covector
\(z\in(\C^6)^{\times }\), \(T^\vee z\) is defined by
\[
T(\Pi_z)=\Pi_{T^\vee z}.
\]
Equivalently, at the linear level,
\[
T^\vee z=(\mc T^{-1})^*(z).
\]

Assume that \(G_f\) is planar of type \(A\), that is, there exists
$
[z]\in A
$
such that
\[
G_f(V)\subset X_z=\Gr_2(\C^4)\cap\Pi_z.
\]
Since \(T\) preserves \(\Gr_2(\C^4)\), we have
\[
T(X_z)
=
T\bigl(\Gr_2(\C^4)\cap\Pi_z\bigr)
=
\Gr_2(\C^4)\cap T(\Pi_z)
=
\Gr_2(\C^4)\cap \Pi_{T^\vee z}
=
X_{T^\vee z}.
\]
By Theorem~\ref{thm:conservative-action},
\[
G_{T^*f}(V)
=
T\cdot G_f(V).
\]
Therefore
\[
G_{T^*f}(V)
\subset
T(X_z)
=
X_{T^\vee z}.
\]
Since \(A\) is a \(G\)-orbit in the dual projective space and \([z]\in A\),
we also have
$
[T^\vee z]\in A.
$
Hence \(G_{T^*f}\) is planar of type \(A\).
\end{proof}

\begin{prop}\label{prop:real-scalars-trivial}
Let $T\in \GL(2,\mb H)$, let $r\in \R^{\times}$, and let $f\in \RS$.
Then $T$ is $f$--admissible  if and only if $rT$ is $f$--admissible, in which case  
\[
(rT)^{*}f=T^{*}f.
\]
\end{prop}

\begin{proof}
If
\[
T=\begin{pmatrix}A&B\\ C&D\end{pmatrix}\, ,
\]
then
\[
rT=\begin{pmatrix}rA&rB\\ rC&rD\end{pmatrix}\, ,
\]
and condition \eqref{eq:admissible-condition} for $rT$ reads 
\[
\det(rA+rB\Phi_{f}(v))\neq 0
\qquad \text{for all }v\in V
\]
which is equivalent to \eqref{eq:admissible-condition} for $T$ because $r\neq 0$.
If the $f$--admissibility condition holds, then
\[
\Phi_{(rT)^{*}f}
=(rC+rD\Phi_{f})(rA+rB\Phi_{f})^{-1}
=(C+D\Phi_{f})(A+B\Phi_{f})^{-1}
=\Phi_{T^{*}f}.
\]
By the uniqueness of the twistor correspondence in the affine chart, $(rT)^{*}f=T^{*}f$.
\end{proof}\begin{cor}
The partial action of $\GL(2,\mb H)$ on slice-regular functions descends to a partial action of $\mr{PGL}(2,\mb H)$.
\end{cor}

\subsection{The globally admissible subgroup}

To obtain a genuine group action on all twistor transforms, we introduce the notion of ``global admissibility''.
\begin{defn}\label{def:global add}
    We say that $T\in \GL(2,\mb H)$ is \emph{globally admissible} if it is $f$--admissible for every $f\in \RS$.
\end{defn}

We first provide some characterizations of global admissibility.
\begin{prop}\label{prop:global-admissible-H}
Let
\[
T=\begin{pmatrix}A&B\\ C&D\end{pmatrix}\in \GL(2,\mb H).
\]
Then the following are equivalent:
\begin{enumerate}[label=\textup{(\roman*)}]
\item $T$ is globally admissible.
\item $T(\mathcal{U}_{6})\subset \mathcal{U}_{6}$.
\item $B=0$.
\end{enumerate}
\end{prop}

To prove Proposition \ref{prop:global-admissible-H}, we need to show that   every point of $\mathcal{U}_{6}$ lies on the twistor transform of some slice-regular polynomial: this is done in the following lemma.
\begin{lem}\label{lem:realize-any-point-new}
For any $v_{0}\in \C\smallsetminus \R$ and for any
\[
\Phi_{0}=\begin{pmatrix}a&b\\ c&d\end{pmatrix}\in \Mat_{2\times 2}(\C)
\]
there exists a slice-regular polynomial $f$ (of degree one) on $\mb H$ such that
\[
\Phi_{f}(v_{0})=\Phi_{0}.
\]

\end{lem}
\begin{proof}
Define two affine holomorphic functions on $\C$ by
\[
g(v):=a\frac{v-\overline v_{0}}{v_{0}-\overline v_{0}}+\overline d\frac{v-v_{0}}{\overline v_{0}-v_{0}},
\qquad
h(v):=c\frac{v-\overline v_{0}}{v_{0}-\overline v_{0}}-\overline b\frac{v-v_{0}}{\overline v_{0}-v_{0}}.
\]
Then
\[
g(v_{0})=a,
\qquad
h(v_{0})=c,
\qquad
\overline{g(\overline v_{0})}=d,
\qquad
\overline{h(\overline v_{0})}=-b.
\]
Let $f$ be the slice-regular polynomial with splitting data $g$, $h$: 
then, by construction,
\[
\Phi_{f}(v_{0})=
\begin{pmatrix}
 g(v_{0}) & -\overline{h(\overline v_{0})}\\
 h(v_{0}) & \overline{g(\overline v_{0})}
\end{pmatrix}
=
\begin{pmatrix}a&b\\ c&d\end{pmatrix}
=\Phi_{0}\, ,
\]
and the final assertion follows from \eqref{lem:graph-chart}.
\end{proof}

We now are able to prove Proposition \ref{prop:global-admissible-H}.
\begin{proof}[Proof of Proposition \ref{prop:global-admissible-H}]
    \textbf{(i)$\Rightarrow$(ii).}
Let $p\in \mathcal{U}_{6}$. By Lemma~\ref{lem:realize-any-point-new}, there exist a slice-regular polynomial $f$ and a point $v_{0}\in V\smallsetminus \R$ such that $G_{f}(v_{0})=p$.
Since $T$ is globally admissible, the curve $T\cdot G_{f}$ is contained in $\mathcal{U}_{6}$, hence
\[
T\cdot p=T\cdot G_{f}(v_{0})\in \mathcal{U}_{6}.
\]
Therefore $T(\mathcal{U}_{6})\subset \mathcal{U}_{6}$.

\textbf{(ii)$\Rightarrow$(iii).}
By Lemmas \ref{lem:graph-chart} and \ref{lem:fractional-linear-graphs}, the condition $T(\mathcal{U}_{6})\subset \mathcal{U}_{6}$ is equivalent to the invertibility of $A+B\Phi$ for every $\Phi\in \Mat_{2\times 2}(\C)$.
Taking $\Phi=0$ gives $A\in \GL(2,\C)$.
Suppose for a contradiction, that $B\neq 0$.
Choose $u\in \C^{2}$ such that $Bu\neq 0$, and set $x:=-A^{-1}Bu$.
Let $\Phi\in \Mat_{2\times 2}(\C)$ be any matrix satisfying $\Phi x=u$.
Then
\[
(A+B\Phi)x=Ax+B\Phi x=-Bu+Bu=0,
\]
contradicting the invertibility of $A+B\Phi$.
Hence $B=0$.

\textbf{(iii)$\Rightarrow$(i).}
If $B=0$, then
\[
A+B\Phi_{f}(v)=A
\]
for every $f$ and every $v$.
Because $T\in \GL(2,\mb H)$, the block $A$ belongs to $\r(\mb H)\cap \GL(2,\C)=\r(\mb H^{\times})$ by \eqref{eq:rho-invertibles}, hence $A$ is invertible.
Therefore, by Remark \ref{rem:admissible for f graph-chart}, $T$ is admissible for every $f$.
\end{proof}

Fix the quaternionic line
\[
\ell:=(0,1)\mathbb H\subset \mathbb H^2,
\]
that is, the projectivization of the second summand in
\[
\mathbb H^2=\mathbb H\oplus \mathbb H.
\]
The natural action of \(\GL(2,\mathbb H)\) on \(\mathbb H^2\) induces a
transitive action on \(\HP^1\).  Hence
\[
\HP^1
\simeq
\GL(2,\mathbb H)/\operatorname{Stab}_{\GL(2,\mathbb H)}(\ell),
\]
where $\mathrm{Stab}(\ell)$ is the subgroup of $\GL(2,\mathbb{H})$ stabilizing $\ell$.

Set
\begin{equation}\label{eq:BH-def}
\mathcal B_{\mathbb H}
:=
\left\{
\begin{pmatrix}
A&0\\
C&D
\end{pmatrix}
\in \GL(2,\mathbb H)
\right\}.
\end{equation}

After projectivization, set
\begin{equation}\label{eq:Gamma-univ-def-new}
\G_{\mathrm{univ}}
:=
\mathcal B_{\mathbb H}/\mathbb R^{\times}
\subset
\PGL(2,\mathbb H).
\end{equation}

\begin{prop}\label{prop:Gamma-subgroup-new}
One has
\[
\mathcal B_{\mathbb H}
=
\operatorname{Stab}_{\GL(2,\mathbb H)}(\ell).
\]
Consequently, \(\mathcal B_{\mathbb H}\) is a subgroup of
\(\GL(2,\mathbb H)\), and
\[
\G_{\mathrm{univ}}
=
\operatorname{Stab}_{\PGL(2,\mathbb H)}(\ell)
\]
is a subgroup of \(\PGL(2,\mathbb H)\).
\end{prop}

\begin{proof}
Let
\[
T=
\begin{pmatrix}
A&B\\
C&D
\end{pmatrix}
\in \GL(2,\mathbb H).
\]
Using the column-vector convention, for every \(q\in\mathbb H\) one has
\[
T
\begin{pmatrix}
0\\
q
\end{pmatrix}
=
\begin{pmatrix}
Bq\\
Dq
\end{pmatrix}.
\]
Therefore \(T\ell\subseteq \ell\) if and only if
\[
Bq=0
\qquad
\text{for all }q\in\mathbb H,
\]
which is equivalent to
\[
B=0.
\]
Since \(T\) is invertible, the condition \(T\ell\subseteq \ell\) is then
equivalent to
\[
T\ell=\ell.
\]
Thus
\[
\operatorname{Stab}_{\GL(2,\mathbb H)}(\ell)
=
\left\{
\begin{pmatrix}
A&0\\
C&D
\end{pmatrix}
\in \GL(2,\mathbb H)
\right\}
=
\mathcal B_{\mathbb H}.\qedhere
\]
\end{proof}

We have the following characterization of the element in $\Gamma_{\mr{univ}}$.
\begin{thm}\label{thm:universal-characterization-new}
Let $[T]\in \mr{PGL}(2,\mb H)$.
The following are equivalent:
\begin{enumerate}[label=\textup{(\roman*)}]
\item $[T]$ contains a globally admissible representative.
\item $[T]\in \G_{\mr{univ}}$.
\end{enumerate}
\end{thm}
\begin{proof}
By Proposition~\ref{prop:global-admissible-H}, a representative
\[
\begin{pmatrix}A&B\\ C&D\end{pmatrix}\in \GL(2,\mb H)
\]
is globally admissible if and only if $B=0$.
This is exactly the condition that the representative belong to $\mc B_{\mb H}$.
Hence (i) and (ii) are equivalent.
\end{proof}

\begin{thm}\label{thm:genuine-action-gamma}
The correspondence
\[
([T],f)\mapsto [T]^*f:= T^{*}f,
\qquad [T]\in \G_{\mr{univ}},
\]
in independent of the choice of a representative $T\in \mc B_{\mb H}$ of $[T]$ and it 
defines a faithful  action of $\G_{\mr{univ}}$ on $\RS$.
\end{thm}
\begin{proof}
The definition is independent of the choice of representative by Proposition~\ref{prop:real-scalars-trivial}, since two representatives of the same class differ by a non-zero real scalar.
Because every element of $\mc B_{\mb H}$ is globally admissible by Proposition~\ref{prop:global-admissible-H}, $[T]^{*}f$ is defined for every $[T]\in \G_{\mr{univ}}$ and every $f$.

Let $[S],[T]\in \G_{\mr{univ}}$ and choose representatives $S,T\in \mc B_{\mb H}$.
Then $ST\in \mc B_{\mb H}$ by Proposition~\ref{prop:Gamma-subgroup-new}, hence all three transforms involved below are defined on the whole class.
By Corollary~\ref{cor:partial-action},
\[
([S][T])^{*}f=(ST)^{*}f=S^{*}(T^{*}f)=[S]^{*}([T]^{*}f).
\]
If $I$ denotes the identity matrix, then $[I]\in \G_{\mr{univ}}$ and
\[
[I]^{*}f=I^{*}f=f.
\]
Therefore $\G_{\mr{univ}}$ acts on the class of slice-regular functions.

We now prove the action is faithful.
Suppose that for a contradiction that $[T]\in \G_{\mr{univ}}$ act trivially on $\RS$.
Choose a representative
\[
T=\begin{pmatrix}A&0\\ C&D\end{pmatrix}\in \mc B_{\mb H},
\]
so that $C\in \r(\mb H)$ and $A,D\in \r(\mb H^\times)$.
Let $a,c,d\in \mb H$ be such that
\[
A=\r(a),\qquad C=\r(c),\qquad D=\r(d).
\]
Now fix any quaternion $p\in \mb H$, and consider the constant slice-regular
function
\[
f_p\equiv p.
\]
From Example \ref{ex:constant-mobius}, one has
\[
[T]^*f_p\equiv (c+d p)a^{-1}.
\]
Since $[T]$ acts trivially, we have $[T]^*f_p=f_p$, hence
\[
(c+d p)a^{-1}=p.
\]
Setting $p=0$ gives $c=0$.
Setting $p=1$ then gives $d=a$.
Hence
\[
ap=pa
\qquad \text{ for all } q\in \mb H.
\]
Thus $a$ belongs to the center of $\mb H$, so $a\in \R^\times$.
It follows that
\[
T=\begin{pmatrix}\r(a)&0\\ 0&\r(a)\end{pmatrix}=a\,I_4,
\]
with $a\in \R^\times$.
Therefore $[T]=[I]$ in $\PGL(2,\mb H)$, and hence also in
$\G_{\mr{univ}}$.
\end{proof}

\subsection{Normal forms for polynomial orbits}

We now restrict the \(\Gamma_{\mathrm{univ}}\)-action to slice-regular polynomials. In this setting the action has an explicit coefficient formula, which yields normal forms and a complete orbit criterion.

In view of Theorem~\ref{thm:genuine-action-gamma} the follwing equivalence relation is well defined.

\begin{defn}\label{def:gamma-orbit-equivalence}
For any $f_1,f_2\in \RS$ we set 
\[
f_1\sim_{\G_{\mr{univ}}}f_2
\]
if and only if  there exists $[T]\in \G_{\mr{univ}}$ such that
\[
f_2=[T]^*f_1.
\]
The corresponding equivalence classes are called the \emph{$\G_{\mr{univ}}$-orbits of  slice-regular functions}.
\end{defn}

In the next Corollary we show that  $\Gamma_{\mr{univ}}$ preserves the class of slice-regular polynomials of degree~$n$. Moreover, if the top power coefficient of a slice-regular polynomial is nonzero, so it is after transformation.
\begin{cor}\label{cor:lower-triangular-polynomial}
Let
$
f(q)=\sum_{m=0}^n  q^m a_m\in \RS
$, and let
$
[T]\in \G_{\mr{univ}}
$ with representative
\[
T=\begin{pmatrix}\r(\a)&0\\ \r (\g)&\r(\d)\end{pmatrix}\in \mc B_{\mb H}.
\]
Then\footnote{The coefficients $u_m$ do not depend on the chosen representative of $[T]$.},
\[
[T]^*f(q)=\sum_{m=0}^n q^m u_m,
\]
where
\begin{equation}\label{eq:coefficients-lower-triangular}
u_0=(\g+\d a_0)\a^{-1},
\qquad
u_m=\d a_m\a^{-1}
\quad (m\ge 1).
\end{equation}
\end{cor}

\begin{proof}
Since $T\in \mc B_{\mb H}$, it has the form
\[
T=\begin{pmatrix}A&0\\ C&D\end{pmatrix},
\qquad
A=\r(\a),\quad C=\r(\g),\quad D=\r(\d).
\]
Hence \eqref{eq:Phi-transform} becomes
\[
\Phi_{[T]^*f}(v)
=
(C+D\Phi_f(v))A^{-1}.
\]
Using Proposition~\ref{prop:polynomial-phi-formula}, we obtain
\[
\Phi_{[T]^*f}(v)
=
\Bigl(C+D\sum_{m=0}^n v^m\r(a_m)\Bigr)A^{-1}
=
(C+D\r(a_0))A^{-1}+\sum_{m=1}^n v^m D\r(a_m)A^{-1}.
\]
Because $\r$ is an $\R$-algebra monomorphism and
\[
A^{-1}=\r(\a^{-1}),
\]
we get
\[
(C+D\r(a_0))A^{-1}
=
\r\bigl((\g+\d a_0)\a^{-1}\bigr),
\]
and, for every $m\ge 1$,
\[
D\r(a_m)A^{-1}
=
\r(\d a_m\a^{-1}).
\]
Hence
\[
\Phi_{[T]^*f}(v)
=
\r(u_0)+\sum_{m=1}^n v^m\r(u_m),
\]
with $u_m$ given by \eqref{eq:coefficients-lower-triangular}. By
\eqref{eq:phi-polynomial-expansion}, this means exactly that
\[
[T]^*f(q)=\sum_{m=0}^n q^m u_m.
\]
Since $\a,\d\in \mb H^\times$ and $a_n\neq 0$, we also have
\[
u_n=\d a_n\a^{-1}\neq 0,
\]
so the degree remains $n$.
\end{proof}

The above result shows that the action is not transitive even on the space of slice-regular polynomials. On the other hand, for each \(n\in\Z_+\), the space of polynomials of degree \(n\) is invariant under this action. This raises a natural question: is the induced action on the space of polynomials of degree \(n\) transitive?

To study this question, we introduce some notation. For $n\ge 1$, let
\[
\mc P_n:=\left\{\sum_{m=0}^n q^m a_m:\ a_m\in \mb H,\ m=0,1,\ldots,n\text{ and }\ a_n\neq 0\right\}
\]
be the set of slice-regular polynomials of degree $n$, and let
\[
\mc N_n:=\left\{q^n+\sum_{m=1}^{n-1} q^m b_m:\ b_m\in \mb H,\ m=1,\ldots,n-1\right\}
\subset \mc P_n
\]
be the set of monic polynomials with vanishing constant term.\par

For \(n\ge1\), define the normalization map
\[
\mathcal N:\mathcal P_n\to \mathcal N_n,\qquad \sum_{m=0}^n q^m a_m\mapsto q^n+\sum_{m=1}^{n-1}q^m a_m a_n^{-1}.
\]
\begin{prop}\label{prop:normalization-Pn}
Every polynomial \(f\in\mathcal P_n\) is \(\Gamma_{\mathrm{univ}}\)-equivalent
to an element of \(\mathcal N_n\).  More precisely, if
\[
f(q)=\sum_{m=0}^n q^m a_m,
\qquad a_n\neq0,
\]
then the element
\[
T_f=
\begin{pmatrix}
a_n&0\\
-a_0&1
\end{pmatrix}
\in\mathcal B_{\mathbb H}
\]
satisfies
\[
[T_f]^*f(q)
=
q^n+\sum_{m=1}^{n-1}q^m a_m a_n^{-1}
=
\mathcal N(f)(q).
\]
\end{prop}

\begin{proof}
Apply Corollary~\ref{cor:lower-triangular-polynomial} with
\[
\alpha=a_n,\qquad
\gamma=-a_0,\qquad
\delta=1.
\]
Then
\[
u_0=(-a_0+a_0)a_n^{-1}=0,
\]
and for \(m\ge1\),
\[
u_m=a_m a_n^{-1}.
\]
In particular,
\[
u_n=a_n a_n^{-1}=1.
\]
Thus
\[
[T_f]^*f(q)=q^n+\sum_{m=1}^{n-1}q^m a_m a_n^{-1}.\qedhere
\]
\end{proof}

\begin{defn}\label{def:simultaneous-conjugacy-Nn}
Two elements
$
q^n+\sum_{m=1}^{n-1}q^m b_m,
q^n+\sum_{m=1}^{n-1}q^m c_m \in\mc N_n
$
are called \emph{simultaneously conjugate} if there exists
$
\eta\in\mathbb H^\times
$
such that
\[
c_m=\eta b_m\eta^{-1}
\qquad\text{ for all }
m=1,\dots,n-1.
\]
\end{defn}
We denote the corresponding quotient by
$
\mathcal N_n/\!\sim.
$
\begin{thm}\label{thm:Pn-orbit-classification}
Let
$
f(q)=\sum_{m=0}^n q^m a_m,
h(q)=\sum_{m=0}^n q^m b_m\in\mc P_n$.  Then \(f\) and \(h\) are
\(\Gamma_{\mathrm{univ}}\)-equivalent if and only if there exists
$
\eta\in\mathbb H^\times
$
such that
\[
b_m b_n^{-1}
=
\eta\,(a_m a_n^{-1})\,\eta^{-1}
\qquad
\text{ for all }
m=1,\dots,n-1.
\]
Equivalently, the map
\[
\mathcal P_n/\Gamma_{\mathrm{univ}}
\longrightarrow
\mathcal N_n/\!\sim,
\qquad
[f]\longmapsto [\mathcal N(f)]
\]
is a bijection.
\end{thm}
\begin{proof}
$(\Rightarrow)$.
Let
\[
T=
\begin{pmatrix}
\rho(\alpha)&0\\
\rho(\gamma)&\rho(\delta)
\end{pmatrix}
\in\mathcal B_{\mathbb H}.
\]
By Corollary~\ref{cor:lower-triangular-polynomial},
\[
[T]^*f(q)=\sum_{m=0}^n q^m u_m,
\]
where
\[
u_0=(\gamma+\delta a_0)\alpha^{-1},
\qquad
u_m=\delta a_m\alpha^{-1}
\quad(m\ge1).
\]
Therefore
\[
u_m u_n^{-1}
=
\delta a_m\alpha^{-1}\alpha a_n^{-1}\delta^{-1}
=
\delta(a_m a_n^{-1})\delta^{-1}
\qquad
(m=1,\dots,n-1).
\]

$(\Leftarrow)$. Suppose there exists \(\eta\in\mathbb H^\times\) such that
\[
b_m b_n^{-1}
=
\eta(a_m a_n^{-1})\eta^{-1}
\qquad
(m=1,\dots,n-1).
\]
Set
\[
\delta:=\eta,
\qquad
\alpha:=b_n^{-1}\eta a_n,
\qquad
\gamma:=b_0\alpha-\eta a_0.
\]
Then \(\alpha,\delta\in\mathbb H^\times\), and
\[
T=
\begin{pmatrix}
\rho(\alpha)&0\\
\rho(\gamma)&\rho(\delta)
\end{pmatrix}
\in\mathcal B_{\mathbb H}.
\]
For \(m=n\), we have
\[
\delta a_n\alpha^{-1}
=
\eta a_n (b_n^{-1}\eta a_n)^{-1}
=
b_n.
\]
For \(1\le m\le n-1\), the assumed simultaneous conjugacy gives
\[
\eta a_m a_n^{-1}\eta^{-1}
=
b_m b_n^{-1}.
\]
Hence
\[
\delta a_m\alpha^{-1}
=
\eta a_m a_n^{-1}\eta^{-1}b_n
=
b_m.
\]
Finally,
\[
(\gamma+\delta a_0)\alpha^{-1}
=
(b_0\alpha-\eta a_0+\eta a_0)\alpha^{-1}
=
b_0.
\]
Therefore
\[
[T]^*f=h.\qedhere
\]
\end{proof}

The preceding proposition shows that every polynomial in \(\mathcal P_n\) can
be normalized to an element of \(\mathcal N_n\), by eliminating the constant
term and making the leading coefficient equal to \(1\).  However, such a
normalized representative is not unique.  Indeed, there may still be elements
of \(\Gamma_{\mathrm{univ}}\) which preserve the normalized class
\(\mathcal N_n\).  The next corollary describes exactly this remaining
freedom: after normalization, the only residual action is the simultaneous
inner conjugation of all intermediate coefficients.
\begin{cor}\label{cor:residual-action-Nn}
After normalization, two elements of \(\mathcal N_n\) represent the same
\(\Gamma_{\mathrm{univ}}\)-orbit precisely when their intermediate
coefficients differ by the same quaternionic conjugation. That is,
\[
q^n+\sum_{m=1}^{n-1}q^m b_m
\sim
q^n+\sum_{m=1}^{n-1}q^m c_m
\]
if and only if there exists \(\eta\in\mathbb H^\times\) such that
\[
c_m=\eta b_m\eta^{-1}
\qquad
(m=1,\dots,n-1).
\]
Moreover, the acting group is
\[
\mathbb H^\times/\mathbb R^\times\simeq SO(3).
\]
\end{cor}
\begin{proof}
Suppose
\[
f(q)=q^n+\sum_{m=1}^{n-1}q^m b_m\in\mathcal N_n
\quad\text{ and }\quad
[T]^*f\in\mathcal N_n,
\qquad
T=
\begin{pmatrix}
\rho(\alpha)&0\\
\rho(\gamma)&\rho(\delta)
\end{pmatrix}.
\]
The condition that the constant term remains zero gives
$
\gamma\alpha^{-1}=0,
$
and hence
$
\gamma=0.
$
The condition that the leading coefficient remains \(1\) gives
$
\delta\alpha^{-1}=1,
$
hence
$
\delta=\alpha.
$
Therefore
\[
c_m=\alpha b_m\alpha^{-1}
\qquad
(m=1,\dots,n-1).
\]

Conversely, every \(\alpha\in\mathbb H^\times\) is realized by
\[
T=
\begin{pmatrix}
\rho(\alpha) &0\\
0&\rho(\alpha)
\end{pmatrix}.
\]

Real non-zero scalars act trivially by conjugation, so the effective residual
group is
\[
\mathbb H^\times/\mathbb R^\times\simeq SO(3).\qedhere
\]
\end{proof}

\begin{rem}
Let \(f(q)=\sum_{m=0}^n q^m a_m\in\mathcal P_n\).  For \(m=1,\dots,n-1\),
set \(c_m:=a_m a_n^{-1}\) and write
\[
c_m=x_m+v_m,\qquad x_m\in\mathbb R,\quad
v_m\in\operatorname{Im}\mathbb H\simeq\mathbb R^3.
\]
Then the \(\Gamma_{\mathrm{univ}}\)-orbit of \(f\) is determined by the real
numbers \(x_1,\dots,x_{n-1}\) and by the diagonal \(SO(3)\)-orbit of
\((v_1,\dots,v_{n-1})\).
\end{rem}

\begin{cor}\label{cor:transitivity-Pn}
The \(\Gamma_{\mathrm{univ}}\)-action on \(\mathcal P_n\) is transitive if
and only if $n=1$.
\end{cor}
\begin{proof}
If \(n=1\), then
$
\mathcal N_1=\{q\}.
$
Hence every degree-one polynomial is \(\Gamma_{\mathrm{univ}}\)-equivalent to
\(q\), so the action is transitive on \(\mathcal P_1\).

If \(n\ge2\), then the two normalized polynomials
\[
q^n
\qquad\text{and}\qquad
q^n+q^{n-1}
\]
are not simultaneously conjugate, since the zero tuple cannot be conjugated
to a tuple with a non-zero component.  Therefore the action on \(\mathcal P_n\)
is not transitive for \(n\ge2\).
\end{proof}

	\bibliographystyle{acm}
	\bibliography{ref}
\end{document}